\documentclass[a4paper,12pt]{article}
\usepackage{dcolumn}
\usepackage{bm}
\usepackage[dvips]{graphicx}
\usepackage{amsmath}
\usepackage{hyperref}
\usepackage{amssymb}
\usepackage{amsfonts}
\usepackage{ucs}
\usepackage[utf8x]{inputenc}
\usepackage[T1]{fontenc}
\usepackage{url}
\usepackage{psfrag}
\usepackage{dsfont}
\usepackage{color}

\newtheorem{proposition}{\textbf{Proposition}}

\newtheorem{remark}{\textbf{Remark}}
\newtheorem{theorem}{\textbf{Theorem}}
\newtheorem{lemma}{\textbf{Lemma}}
\newtheorem{definition}{\textbf{Definition}}

\def\dx{{\rm d}x}

\def\del  {\partial}
\def\eps{\varepsilon}

\def\R{\mathbb{R}}

\def\Z{\mathbb{Z}}
\def\N{\mathbb{N}}

\def\dx{{\rm d}x}

\def\dl{{\rm d}\tau}

\newenvironment{proof}{%
{\noindent \bf Proof : }%
}{%
\hfill$\Box$\\%
}
\begin{document}
\title{Cauchy Problem for the Kuznetsov Equation}

\author{ADRIEN DEKKERS\footnote{adrien.dekkers@centralesupelec.fr} and  ANNA ROZANOVA-PIERRAT\footnote{Laboratoire Math\'ematiques et Informatique Pour la Complexit\'e et les Syst\`emes,
Centrale Sup\'elec, Universit\'e Paris-Saclay, Campus de Gif-sur-Yvette, Plateau de Moulon, 
3 rue Joliot Curie, 91190 Gif-sur-Yvette,France, 
anna.rozanova-pierrat@centralesupelec.fr}}

\maketitle

\begin{abstract}
We consider the Cauchy problem for a model of non-linear acoustic, named the Kuznetsov equation, describing a sound propagation in thermo-viscous elastic media. For the viscous case, it is a weakly quasi-linear strongly damped wave equation, for which we prove the global existence in time of regular solutions for sufficiently small initial data, the size of which is specified, and give the corresponding energy estimates.
In the inviscid case, we update the known results of John for quasi-linear wave equations, obtaining the well-posedness results for less regular initial data. We obtain, using a priori estimates and a Klainerman inequality, the estimations of the maximal existence time, depending on the space dimension, which are optimal, thanks to the blow-up results of Alinhac.  Alinhac's blow-up results are also confirmed by a $L^2$-stability estimate, obtained between a regular and a less regular solutions.
\end{abstract}
\section{Introduction}\label{intro}
The Kuznetsov equation~\cite{Kuznetsov} models a propagation of non-linear acoustic waves in thermo-viscous elastic media.
This equation describes the evolution of the velocity potential and can be derived, as in~\cite{Roz1}, from a compressible isentropic
Navier-Stokes system, for the viscous case, or the Euler system for the inviscid case, using small perturbations of the density and of the velocity characterized by a small dimensionless parameter $\varepsilon>0$. The Cauchy problem for the Kuznetsov equation  reads for  $\alpha=\frac{\gamma-1}{c^2}$, $\beta=2$ and $\nu= \frac{\delta}{\rho_0}$ as
\begin{align}
& u_{tt}-c^2\Delta u- \nu \varepsilon \Delta u_t=\alpha\varepsilon u_t  u_{tt}+\beta \varepsilon \nabla u \;\nabla u_t,\;\;\;\;x\in\mathbb{R}^n, \label{kuz}\\
& u(x,0)=u_0(x),\quad u_t(x,0)=u_1(x), \quad x\in\mathbb{R}^n,\label{ci}
\end{align}
where $c$, $\rho_0$, $\gamma$, $\delta$ are the velocity of the sound, the density, the ratio of the specific heats and the
viscosity of the medium respectively. In what follows, we can just suppose that $\alpha$ and $\beta$ are some positive constants.
Eq. (\ref{kuz}) is a weakly quasi-linear damped wave equation, that describes a propagation of a high amplitude wave in fluids.
The Kuznetsov equation is one of the models derived from the Navier-Stokes system, and it is well suited for the plane, cylindrical and spherical waves in a fluid~\cite{Hamilton}. Most of the works on the Kuznetsov equation~(\ref{kuz}) are treated in the one space dimension~\cite{Jordan} or in a bounded spatial domain of $\R^n$~\cite{Kalt2,Kalt1,Meyer}.
For the viscous case Kaltenbacher and Lasiecka~\cite{Kalt1} have considered the Dirichlet boundary valued problem and proved for sufficiently small  initial data the global well-posedness
  for $n\leq 3$. Meyer and Wilke~\cite{Meyer} have proved it for all $n$. In~\cite{Kalt2} it was proven a local well-posedness of the Neumann boundary valued problem  for $n\leq 3$.

In this article we study the well-posedness properties of the Cauchy problem~(\ref{kuz})--(\ref{ci}).
In the inviscid case for $\nu=0$, the Cauchy problem for the Kuznetsov equation is a particular case of a general quasi-linear hyperbolic system of the second order considered by  Hughes, Kato and Marsden~\cite{Kato} (see Theorem~\ref{ThMainWPnu0} Points~1 and~2 for the application of their results to the Kuznetsov equation). The local well-posedness result, proved in~\cite{Kato}, does not use  a priori estimate techniques and is based on the semi-group theory. Hence, thanks to~\cite{Kato},  we have the well-posedness of~(\ref{kuz})--(\ref{ci}) in the Sobolev spaces $H^s$ with  a real $s>\frac{n}{2}+1$.  Therefore, actually, to extend the local well-posedness to a global one (for $n\ge 4$) and to estimate the maximal time interval on which there exists a regular solution, John~\cite{John} has developed a priori estimates for the Cauchy problem for a  general quasi-linear wave equation. This time, due to the non-linearities $u_t  u_{tt}$ and $\nabla u \;\nabla u_t$ including the time derivatives, to have an a priori estimate for the Kuznetsov equation we need to work with Sobolev spaces with a natural $s$, thus denoted in what follows  by $m$.
If we directly apply general results of Ref.~\cite{John} to our case of  the Kuznetsov equation, we obtain a well-posedness result with a  high regularity of the initial data.
We improve it in Theorem~\ref{ThTEt1} and show John's results for the Kuznetsov equation with the minimal regularity on the initial data corresponding to the regularity obtained by Hughes, Kato and Marsden~\cite{Kato}.  For instance, we prove the analogous energy estimates in $H^m$ with $m\geq [\frac{n}{2}+2]$ instead of John's $m\geq \frac{3}{2}n+4$ (see Eq.~(\ref{EqENfirstNV}) in Proposition~\ref{PropApEstNV1}) and its slight modified version in $H^m$ with $m\geq  [\frac{n}{2}+3]$ instead of $m\geq \frac{3}{2}n+7$ (see Eq.~(\ref{EqEnSecondNV}) in Proposition~\ref{PropApEstNV2}).
The energy estimates allow us to evaluate the maximal existence time  interval (see Theorem~\ref{ThMainWPnu0} Point~5 and Theorem~\ref{ThTimeExistNu0} for more details).
In $\R^2$ and $\R^3$ the optimality of  obtained estimations for the maximal existence time is ensured by the results of  Alinhac~\cite{Alinhac}. In Ref.~\cite{Alinhac} a geometric blow-up for small data  is proved for $\del_t^2u$ and $\Delta u$ at a finite time of the same order as predicted by our a priori estimates (see Theorem~\ref{ThMainWPnu0} Point~5, our estimates of the minimum existence time correspond to Alinhac's maximum existence time results).
From the other hand, the blow-up of $\del_t^2u$ and $\Delta u$ is also confirmed by the stability estimate~(\ref{StabEstNV}) in Theorem~\ref{ThMainWPnu0}: if the maximal existence time interval is finite and limited by $T^*$, by Eq.~(\ref{StabEstNV}), we have the divergence
\begin{equation}\label{EqDivInt}
 \int_0^{T^*}\left( \Vert u_{tt}\Vert_{L^{\infty}(\mathbb{R}^n)}+ \Vert \Delta u\Vert_{L^{\infty}(\mathbb{R}^n)}\right)\dl=+\infty.
\end{equation}
For $n\ge4$ and $\nu=0$, we also improve the results of John~\cite{John}  and show the global existence for sufficiently small initial data  $u_0\in H^{m+1}(\mathbb{R}^n)$ and $u_1\in H^m(\mathbb{R}^n)$ with $m\ge  n+2$ instead of $m\geq \frac{3}{2}n+7$ (see Proposition~\ref{Propglobapenu0} and Theorem~\ref{ThTimeExistNu0}). The smallness of the initial data here directly ensures the hyperbolicity of the Kuznetsov equation for all time, $i.e.$ it ensures that $1-\alpha\varepsilon u_t$ is strictly positive and bounded for all time.
The proof uses the generalized derivatives for the wave type equations~\cite{John} and a priori estimate of Klainerman~\cite{Klainerman1, Klainerman2} (see Section~\ref{SecJohnn4}).

Let us now formulate our main well-posedness result for the inviscid case:
\begin{theorem}\textbf{(Inviscid case)}\label{ThMainWPnu0}
  Let $\nu=0$, $n\in \mathbb{N}^*$ and $s>\frac{n}{2}+1$. For all $u_0\in H^{s+1}(\mathbb{R}^n)$ and $u_1\in H^s(\mathbb{R}^n)$ such that $\Vert u_1\Vert_{L^{\infty}(\mathbb{R}^n)}< \frac{1}{2 \alpha\varepsilon}$, $\Vert u_0\Vert_{L^{\infty}(\mathbb{R}^n)}<M_1$, $\Vert \nabla u_0\Vert_{L^{\infty}(\mathbb{R}^n)}<M_2$, with $M_1$ and $M_2$ in $\mathbb{R}^*_+$ the following results hold:
 \begin{enumerate}
  \item For all $T>0$, there exists $T'>0$, $T'\le T$, such that there exists a unique solution $u$ of~(\ref{kuz})--(\ref{ci}) with the following regularity
\begin{align}
u\in C^r([0,T'];H^{s+1-r}(\mathbb{R}^n))\;\;\text{for}\;\;0\leq r\leq s,\label{regkuznu0}\\
\forall t\in [0,T'],\;\;\Vert u_t(t)\Vert_{L^{\infty}(\mathbb{R}^n)}< \frac{1}{2 \alpha\varepsilon},\;\;\Vert u\Vert_{L^{\infty}(\mathbb{R}^n)}<M_1,\;\;\Vert \nabla u\Vert_{L^{\infty}(\mathbb{R}^n)}<M_2.\label{conshyper}
\end{align}
\item The map $(u_0,u_1)\mapsto (u(t,.),\partial_t u(t,.))$ is continuous in the topology of $H^{s+1}\times H^s$ uniformly in $t\in[0,T']$.
\item Let $T^*$ be the largest time on which such a solution is
defined, and in addition $s\in \N$, $i.e.$ $s=m\ge m_0=[\frac{n}{2}+2]$. With the notation
\begin{equation}\label{EnergEN}
        E_{m}[u](t)=\Vert \nabla u(t)\Vert_{H^m(\mathbb{R}^n)}^2+\sum_{i=1}^{m+1} \Vert \partial_t^i u(t)\Vert_{H^{m+1-i}(\mathbb{R}^n)}^2,
       \end{equation}
there  exist  constants $C(n,c,\alpha)>0$ and $\hat{C}(n,c,\alpha,\beta)>0$ (see Theorem~\ref{ThTEt1}) such that if the initial data satisfies $\sqrt{E_{m_0}[u](0)}\le \frac{1}{C(n,c,\alpha)\eps}$,  then
\begin{equation}\label{etoile}
T^*\ge \frac{1}{\eps \hat{C}(n,c,\alpha,\beta) \sqrt{E_{m_0}[u](0)}}, \hbox{ such that it holds~(\ref{EqDivInt}).}
\end{equation}
\item  For two solutions $u$ and $v$ of the Kuznetsov equation for $\nu=0$ defined on $[0,T^*[$ assume that  $u$ be regular as in~(\ref{regkuznu0})--(\ref{conshyper}), $i.e.$
$u\in L^{\infty}([0,T^*[;H^{m+1}(\mathbb{R}^n))$, $u_t \in L^{\infty}([0,T^*[;H^{m}(\mathbb{R}^n))$ ($s=m$ as in Point~3), and
 $$v\in L^{\infty}([0,T^*[;H^1(\mathbb{R}^n)),\quad v_t\in L^{\infty}([0,T^*[;L^2(\mathbb{R}^n))\hbox{ with }\Vert v\Vert_{L^{\infty}(\mathbb{R}^n)}< \frac{1}{2 \alpha\varepsilon}$$ and with a bounded $\|\nabla v_t\|_{L^\infty(\mathbb{R}^n)}$ norm on $[0,T^*[$.
Then it holds the following stability uniqueness result: there exist constants $C_1>0$ and $C_2>0$, independent on time, such that
\begin{multline}
(\Vert (u-v)_t\Vert_{L^2}^2+ \Vert \nabla(u-v)\Vert_{L^2}^2)(t)
\leq C_1 \exp\left(C_2\varepsilon \int_0^t\sup( \Vert u_{tt}\Vert_{L^{\infty}(\mathbb{R}^n)}, \Vert \Delta u\Vert_{L^{\infty}(\mathbb{R}^n)})\dl\right)\\
.(\Vert u_1-v_1\Vert_{L^2}^2+ \Vert \nabla(u_0-v_0)\Vert_{L^2}^2).\label{StabEstNV}
\end{multline}
 \item If  $s=m\ge n+2$, then for  sufficiently small initial data (see Theorem~\ref{ThTimeExistNu0} in Section~\ref{SecJohnn4})
\begin{enumerate}
             \item  $\liminf_{\varepsilon\rightarrow 0} \varepsilon^2 T^* >0 \;\;\;\text{for}\;n=2,$
\item $\liminf_{\varepsilon\rightarrow 0} \varepsilon \log(T^*)>0\;\;\;\text{for}\;n=3,$
\item $T^*=+\infty$ for $n\ge 4$.
            \end{enumerate}
                        \end{enumerate}
\end{theorem}
Theorem~\ref{ThMainWPnu0} is principally based on the a priori estimates given in Sections~\ref{SecKato} (for Point~3) and~\ref{SecJohnn4} (for Point~5) and on the local existence result updated from Ref.~\cite{Kato} (Points 1 and 2).  Point~4  uses the classical ideas of the weak-strong stability, for instance proved in details for the KZK equation in~\cite{Roz2} Theorem~1.1 Point~4 p.~785. Hence its proof is omitted. Some technical details on the proof of the a priori estimates of Section~\ref{SecKato} can be found in Appendix~\ref{Appen1}.

Analyzing the structure of the Kuznetsov equation and the difficulties involving by its non-linear terms, we start in Section~\ref{SecENL2} with preliminary remarks on the $L^2$-energy properties for the Kuznetsov equation to compare with its simplified versions. Developing the energy estimates in the Sobolev spaces, we however recognize the structure of the $L^2$-energy of the wave equation which keeps unchanged.

In the presence of the term $\Delta u_t$ for the viscous case $\nu>0$, the regularity of the higher order time derivatives of $u$ is different (to compare to the inviscid case), and the way to control the non-linearities in the a priori estimates becomes different. As it was shown in~\cite{Shibata}, this dissipative term changes a finite speed of propagation of the wave equation to the infinite one. Indeed, the linear part of Eq.~(\ref{kuz}) can be viewed as two compositions of the heat operator $\del_t-\Delta$ in the following way:
$$u_{tt}-c^2\Delta u- \nu \varepsilon \Delta u_t=\del_t(\del_t u -\eps \nu \Delta u)-c^2 \Delta u.$$

For the viscous case we prove the global in time well-posedness results in $\R^n$  (see  Section~\ref{secV}) for small enough initial data, the size of which we specify (see Point~1 of Theorem~\ref{ThMainWPnuPGlob} and Subsection~\ref{secVPr} for its proof). In Subsection~\ref{sec4} for $n\ge 3$ (see Point~2 of Theorem~\ref{ThMainWPnuPGlob}) we establish an a priori estimate which gives also a sufficient condition of the existence of a global solution for a sufficiently small initial energy of the same order on $\eps$ as in Point~1 of Theorem~\ref{ThMainWPnuPGlob}. The same results (see Point~3 of Theorem~\ref{ThMainWPnuPGlob}) hold in $(\R/L\Z)\times \R^{n-1}$ for $n\ge 2$ (with a periodicity and mean value zero on one variable).

\begin{theorem}\textbf{(Viscous case)}\label{ThMainWPnuPGlob}
 Let $\nu>0$, $n\in \mathbb{N}^*$, $s>\frac{n}{2}$ and $\R^+=[0,+\infty[$. Considering the Cauchy problem for the Kuznetsov equation~(\ref{kuz})--(\ref{ci}), the following results hold:
 \begin{enumerate}
  \item Let   $$X:=H^2(\R^+;H^{s}(\mathbb{R}^n))\cap H^1(\R^+;H^{s+2}(\mathbb{R}^n)),$$ the initial data
  $$u_0\in H^{s+2}(\mathbb{R}^n)\quad  \hbox{and}\quad u_1\in H^{s+1}(\mathbb{R}^n),$$
  $r_*=O(1)$ be the positive constant defined in  Eq.~(\ref{Eqret}) and $C_1=O(1)$ be the minimal constant such that the
  solution $u^*$ of the corresponding linear Cauchy problem (\ref{kuzlin0inhom}) satisfies
  $$\|u^*\|_X\le \frac{C_1}{\sqrt{\nu \eps}}(\Vert u_0\Vert_{H^{s+2}(\mathbb{R}^n)}+\Vert u_1\Vert_{H^{s+1}(\mathbb{R}^n)}).$$

 Then  for all  $r\in[0,r_{*}[$
  and all
  initial data  satisfying
\begin{equation}\label{EqInDSmal1}
  \Vert u_0\Vert_{H^{s+2}(\mathbb{R}^n)}+\Vert u_1\Vert_{H^{s+1}(\mathbb{R}^n)}\le \frac{\sqrt{\nu \eps}}{C_1}r,
  \end{equation}
there exists the unique solution $u\in X$  of the Cauchy problem for the Kuznetsov equation and $\Vert u\Vert_X\leq 2 r$.

  \item Let $n\ge 3$, $s=m\in \N$ be even and $m\geq [\frac{n}{2}+3]$.  With the notation
\begin{equation}\label{EnergEN2}
        E_{\frac{m}{2}}[u](t)=\Vert \nabla u(t)\Vert_{H^m(\mathbb{R}^n)}^2+\sum_{i=1}^{\frac{m}{2}+1} \Vert \partial_t^i u(t)\Vert_{H^{m-2(i-1)}(\mathbb{R}^n)}^2,
       \end{equation}
there  exists a constant $\rho=O(1)>0$ (see Theorem~\ref{ThDecrEnnu} Point~2), independent on time, such that for all initial data $u_0\in H^{m+1}(\mathbb{R}^n)$ and $u_1\in H^m(\mathbb{R}^m)$ satisfying
\begin{equation}\label{EqInDSmal2}
   E_{\frac{m}{2}}[u](0)<\rho\eps,                                                                                                                                                            \end{equation}
  there exists a unique $u\in C^{0}(\R^+;H^{m+1}(\mathbb{R}^n))\cap C^i(\R^+;H^{m+2-2i}(\mathbb{R}^n)),$ for $i=1,..,\frac{m}{2}+1$ with the bounded energy
$$\forall t\in \mathbb{R}^+,\;\;\; E_{\frac{m}{2}}[u](t)\leq O\left(\frac{1}{\eps}\right)\; E_{\frac{m}{2}}[u](0)=O(1).$$
\item For $n\in \N^*$ in $\Omega=(\R/L\Z)\times \R^{n-1}$ with $s=m\in \N$  even and $m\geq [\frac{n}{2}+3]$ there hold Points~1 and~2 in the class of periodic in one direction functions with the mean value zero
\begin{equation}\label{percond}
\int_{\R/L\Z} u(t,x,y)\; \dx=0.
\end{equation}
\end{enumerate}
\end{theorem}
Let us notice that the hyperbolicity condition~(\ref{conshyper}) is also satisfied if we require conditions~(\ref{EqInDSmal1}) and~(\ref{EqInDSmal2}). For $\nu>0$ Point~4 of Theorem~\ref{ThMainWPnu0} obviously holds for all $n\in \N^*$. Point~1 of Theorem~\ref{ThMainWPnuPGlob} is proved in Subsection~\ref{secVPr} using a theorem of a non-linear analysis~\cite{Sukhinin} (see Theorem~\ref{thSuh}) and regularity results for the strongly damped wave equation following~\cite{Haraux}, which can also be used for $\Omega=(\R/L\Z)\times \R^{n-1} $ in point~3.
Point~2 of Theorem~\ref{ThMainWPnuPGlob} is proved in Subsection~\ref{sec4}, using  a priori estimates given in Proposition~\ref{PropEnergiVisc}, see also~Theorem~\ref{ThDecrEnnu}.
The last point of Theorem~\ref{ThMainWPnuPGlob} is a direct corollary of the Poincar\'{e} inequality
\begin{equation}\label{Poincar}
 \Vert u\Vert_{L^2((\R/LZ)\times \R^{n-1})}\leq C \Vert \partial_x u\Vert_{L^2((\R/LZ)\times \R^{n-1})},
\end{equation}
which holds in the class of periodic functions with the mean value zero.
Estimate~(\ref{Poincar}) allows to have the same estimate as in Lemma~\ref{PropEnergiVisc} (see Section~\ref{secV}) for $n=2$, which fails in $\R^2$.
Thus, it also gives the global existence for rather small initial data detailed in Point~2.

\section{Preliminary remarks on $L^2$-energies}\label{SecENL2}
We can notice that Eq.~(\ref{kuz}) is a wave equation containing a dissipative term $\Delta u_t$ and two non-linear terms:  $\nabla u\nabla u_t$ describing local non-linear effects and $u_t u_{tt}$ describing  global or cumulative effects. 
Actually, the linear wave equation appears from Eq.~(\ref{kuz}) if we consider only the terms of the zero order on $\varepsilon$:
\begin{equation}\label{wave}
u_{tt}-c^2\Delta u =0.
\end{equation}
The semi-group theory permits in the usual  way  to show that for $u_0\in H^1(\mathbb{R}^n)$ and $u_1\in L^2(\mathbb{R}^n)$ there exists a unique solution of the Cauchy problem~(\ref{wave}),~(\ref{ci})
$$u\in C^0(\R^+;H^1(\mathbb{R}^n))\cap C^1(\R^+;L^2(\mathbb{R}^n)).$$So the energy of the wave equation~(\ref{wave})
\begin{equation}\label{enwave}
E(t)=\int_{\mathbb{R}^n} [(u_t)^2+c^2(\nabla u)^2](t,x)\dx,
\end{equation}
is well defined and  conserved
$$\frac{d}{dt}E(t)=0.$$

For $\nu>0$ and without non-linear terms, the Kuznetsov equation~(\ref{kuz}) becomes the known strongly damped wave equation:
\begin{equation}\label{dampwave}
u_{tt}-c^2\Delta u- \nu \varepsilon \Delta u_t =0,
\end{equation}
which is well-posed~\cite{Ikehata}: for $m\in\mathbb{N}$, $u_0\in H^{m+1}(\mathbb{R}^n)$ and $u_1\in H^{m}(\mathbb{R}^n)$ there exists a unique solution of the Cauchy problem~(\ref{dampwave}),~(\ref{ci})
$$u\in C^0(\R^+;H^{m+1}(\mathbb{R}^n))\cap C^1(\R^+;H^m(\mathbb{R}^n)).$$

Multiplying Eq.~(\ref{dampwave}) by $u_t$ in $L^2(\mathbb{R}^n)$, we obtain for the energy of the wave equation~(\ref{enwave})
$$ \frac{d}{dt}E(t)= -2\nu \varepsilon  \int_{\mathbb{R}^n}(\nabla u_t)^2(t,x)\dx\leq 0,$$
 what means that the energy $E(t)$  decreases in time, thanks to the viscosity term with $\nu>0$.
 The decrease rate is found for more regular energies in~\cite{Shibata}  in accordance with the regularity of the initial conditions.
 Without the term $\nabla u\nabla u_t$  (local non-linear effects), the Kuznetsov equation becomes similar to the Westervelt equation, initially derived by Westervelt~\cite{Westervelt} before Kuznetsov.
 The Westervelt equation, historically derived~\cite{Westervelt} for the acoustic pressure fluctuation, has the following form
 \begin{align}
& p_{tt}-c^2\Delta p - \nu \varepsilon \Delta p_t=\frac{\gamma+1}{c^2}\varepsilon p_t  p_{tt},\label{Westervelt}
\end{align}
and can also be seen as an approximation of an isentropic Navier-Stokes system.

In the sequel we  conveniently denote $p$ by $u$.
We multiply Eq.~(\ref{Westervelt}) by $u_t$ and integrate over $\mathbb{R}^n$ to obtain
$$\frac{1}{2}\frac{d}{dt}\left(\int_{\mathbb{R}^n} [(u_t)^2+c^2 (\nabla u)^2]\;\dx\right)+\nu \varepsilon \int_{\mathbb{R}^n} (\nabla u_t)^2\;\dx=\frac{1}{3} \frac{\gamma+1}{c^2}\varepsilon \frac{d}{dt}\left(\int_{\mathbb{R}^n} (u_t)^3\;\dx\right).$$
Then we have
$$\frac{1}{2}\frac{d}{dt}\left(\int_{\mathbb{R}^n}\left[\left(1-\frac{2}{3} \frac{\gamma+1}{c^2}\varepsilon u_t\right) (u_t)^2+c^2 (\nabla u)^2\right]\;\dx\right)+\nu \varepsilon \int_{\mathbb{R}^n} (\nabla u_t)^2\;\dx=0.$$
For $\alpha =\frac{2}{3} \frac{\gamma+1}{c^2}$ we consider the energy
\begin{equation}
 \label{EqE1}
 E_{nonl}(t)=\int_{\mathbb{R}^n}\left[\left(1-\alpha\varepsilon u_t\right) (u_t)^2+c^2 (\nabla u)^2\right]\;\dx,
\end{equation}
which is monotonous decreasing for $\nu>0$ and is conserved for $\nu=0$.
Let us also notice that, taking the same initial data for $\nu=0$ and $\nu>0$, we have:
$$\hbox{for all } \nu>0 \hbox{ and } t>0\quad E_{nonl}(t,\nu=0)> E_{nonl} (t,\nu)\geq 0,$$
in the assumption that $1-\alpha \varepsilon u_t\geq 0$ almost everywhere.

While
$\frac{1}{2}\leq 1-\alpha\varepsilon u_t\leq \frac{3}{2}$, that is to say $\Vert u_t(t)\Vert_{L^\infty(\mathbb{R}^n)}$ remains small enough in time, then we can compare $E_{nonl}$ to the  energy of the wave equation
$$\frac{1}{2}E(t)\leq E_{nonl}(t)\leq \frac{3}{2}E(t).$$
Then a sufficiently regular solution of the Cauchy problem for the Westervelt equation  has the energy $E$ controlled by a decreasing in time function:
$$
E(t)\leq 3 E(0)-4\nu \varepsilon \int_0^t \int_{\mathbb{R}^n} (\nabla u_t(\tau,x))^2\dx\;\dl.
$$

Now, let us consider the Kuznetsov equation~(\ref{kuz}). We multiply it by $u_t$ and integrate on $\mathbb{R}^n$ to obtain
$$\frac{1}{2}\frac{d}{dt}E_{nonl}(t)+\nu \varepsilon \int_{\mathbb{R}^n} (\nabla u_t)^2\;\dx=2\varepsilon \int_{\mathbb{R}^n} \nabla u\;\nabla u_t\; u_t\;\dx,$$
where $E_{nonl}(t)$ is given by Eq.~(\ref{EqE1}) with  $\alpha=\frac{2}{3} \frac{\gamma-1}{c^2}$.
As $$2\eps \int_{\mathbb{R}^n} \nabla u\;\nabla u_t\; u_t\;\dx=\eps\frac{d}{dt}\int_{\R^n}u_t(\nabla u)^2\;\dx-\eps\int_{\R^n}u_{tt}(\nabla u)^2\; \dx,$$
we find
\begin{multline}
 \frac{1}{2}\frac{d}{dt}\left(\int_{\mathbb{R}^n}\left[\left(1-\frac{2}{3} \frac{\gamma-1}{c^2}\varepsilon u_t\right) (u_t)^2+(c^2-2\eps u_t) (\nabla u)^2\right]\;\dx\right.\\+ \left. 2\eps\int^t_0\int_{\R^n}u_{tt}|\nabla u|^2\;\dx\;\dl\right)+ \nu \varepsilon \int_{\mathbb{R}^n} (\nabla u_t)^2\;\dx=0.\label{L2KuzEnFull}
\end{multline}
Thus, for $\alpha=\frac{2}{3} \frac{\gamma-1}{c^2}$, the function $$F_{\nu}(t)=\int_{\mathbb{R}^n}\left[\left(1-\alpha\varepsilon u_t\right) (u_t)^2+(c^2-2\eps u_t) (\nabla u)^2\right]\;\dx+ 2\eps\int^t_0\int_{\R^n}u_{tt}|\nabla u|^2\;\dx\;\dl$$
is constant if $\nu=0$ and decreases if $\nu>0$. Let us notice that while
$\frac{1}{2}\leq 1-\alpha\varepsilon u_t\leq \frac{3}{2}$, the coefficient $c^2-2\eps u_t$ is always positive (since $c$ is the sound speed in the chosen medium, $c^2\gg 1$), hence the first integral in $F_{\nu}(t)$ is positive, but we a priori don't know the sign of the second integral, $i.e.$ the sign of $u_{tt}$. However, for $\nu=0$, $F_{\nu=0}(t)$ is positive, as soon as $0\leq 1-\alpha \varepsilon u_1$:
$$F_{\nu=0}(t)=F_{\nu=0}(0)=\int_{\mathbb{R}^n}\left[\left(1-\alpha\varepsilon u_1\right) (u_1)^2+(c^2-2\eps u_1) (\nabla u_0)^2\right]\;\dx\ge0,$$
and, if we take the same initial data for the Cauchy problems with $\nu=0$ and $\nu>0$, for all $t> 0$ (for all time where $F_{\nu=0}$ exists)  it holds $F_{\nu=0}(t)=F_{\nu=0}(0)>F_{\nu>0}(t)$.

For $n\geq3$, we can control the term $2\varepsilon \int_{\mathbb{R}^n} \nabla u\nabla u_t u_t\;\dx$ using the H\"{o}lder inequality and the Sobolev embeddings (which fails in $\R^2$):
\begin{align*}
\left\vert \int_{\mathbb{R}^n} \nabla u\;\nabla u_t\; u_t\;\dx\right\vert\leq & \Vert\nabla u\Vert_{L^n}\Vert \nabla u_t\Vert_{L^2}\Vert u_t\Vert_{L^{\frac{2n}{n-2}}}
\leq  C \Vert\nabla u\Vert_{L^n}\Vert \nabla u_t\Vert_{L^2}^2.
\end{align*}
Indeed, in $\mathbb{R}^2$ we don't have any estimates of the form
$$\Vert u\Vert_{L^p(\mathbb{R}^2)}\leq \Vert \nabla u\Vert_{L^2(\mathbb{R}^2)}, $$
    with $p>2$. But  such an estimate is essential to control the nonlinear term.
 Then, instead of Eq.~(\ref{L2KuzEnFull}) for $F_{\nu}$, we have the relation for $E_{nonl}$:
$$\frac{1}{2}\frac{d}{dt}E_{nonl}(t)+(\nu \varepsilon-2\varepsilon C \Vert\nabla u\Vert_{L^n}) \int_{\mathbb{R}^n} (\nabla u_t)^2\;\dx\le 0.$$
So, if a solution of the Kuznetsov equation $u$ is such that $\Vert \nabla u(t)\Vert_{L^n}$ and $\Vert u_t(t)\Vert_{L^{\infty}}$ stay small enough for all time, then $E_{nonl}$ decreases in time and, as previously for the Westervelt equation, thanks to
$\frac{1}{2}E(t)\leq E_{nonl}(t)\leq \frac{3}{2}E(t)$, the energy $E$ has for  upper bound a decreasing function.

This fact leads us to look for global well-posedness results for the Cauchy problem for the Kuznetsov equation in the viscous case.

\section{Well-posedness for the inviscid case}\label{sec2}
\subsection{Proof of Point~3 of Theorem~\ref{ThMainWPnu0}}\label{SecKato}
Let us give an estimation of the maximum existence time for a solution of problem~(\ref{kuz})--(\ref{ci}) with $\nu=0$. For this we follow the work of John~\cite{John} with the use of a priori estimate. However we don't directly apply the general results of John, but we  improve them for our specific problem as we can take less regular initial conditions in order to have suitable a priori estimates.

\begin{proposition}~\label{PropApEstNV1}
For a fixed $m\in \mathbb{N}$ with $m\geq m_0=\left[\frac{n}{2}+2\right]$, let $u$ be a local solution of problem~(\ref{kuz})--(\ref{ci}) with $\nu=0$ on $[0,T]$ satisfying~(\ref{regkuznu0}) and~(\ref{conshyper}) for $s=m$.

For $t\in [0,T]$  we have for $E_{m}[u](t)$, defined in Eq.~(\ref{EnergEN}),
\begin{equation}\label{EqENfirstNV}
E_{m}[u](t)\leq B\; E_{m}[u](0)+C_m \max(\alpha,\beta)\varepsilon \int_{0}^{t}E_{m}[u](\tau)^{\frac{3}{2}} \dl,
\end{equation}
with constants $B=\frac{(3+2c^2)}{\min(1/2,c^2)}>0$, depending only on $c$, and $C_m>0$, depending only on $m$, on the dimension $n$ and on $c$ (only if $\min(1/2,c^2)=c^2$).
\end{proposition}
\begin{proof}
The proof is given in Appendix~\ref{Appen1}.
\end{proof}

Inequality~(\ref{EqENfirstNV}), proved in Proposition~\ref{PropApEstNV1}, gives us an a priori estimate in order to have, with the help of the Gronwall Lemma, an estimation of the maximum existence time $T^*$.  However,  when $m$ increases, $C_m$ increases, and the maximum existence time, given by estimate~(\ref{EqENfirstNV}),  decreases whereas the initial conditions become more regular. Therefore,
we prove the second a priori estimate (see Eq.~(\ref{EqEnSecondNV})), playing a key role in order to avoid this problem:

\begin{proposition}~\label{PropApEstNV2}
Let conditions of Proposition~\ref{PropApEstNV1} be satisfied. Then for $t\in [0,T]$ and $m\geq \left[\frac{n}{2}+3\right]$ we have
\begin{equation}\label{EqEnSecondNV}
E_{m}[u](t)\leq B\; E_{m}[u](0)+D_m\max(\alpha,\beta) \varepsilon \int_{0}^{t}E_{m-1}[u](\tau)^{\frac{1}{2}}E_{m}[u](\tau) \dl,
\end{equation}
with a constant $D_m>0$, depending only on $m$, on $n$ and on $c$ and the same constant $B$ as in Proposition~\ref{PropApEstNV1}.
\end{proposition}
The proof  of Eq.~(\ref{EqEnSecondNV}) is very similar to the proof of Proposition~\ref{PropApEstNV1} given in Appendix~\ref{Appen1} and hence omitted (see Remark~\ref{RemAprop2} in  Appendix~\ref{Appen1}).

Now let us give a first estimation of the lifespan $T^*$ of a local solution of problem~(\ref{kuz})--(\ref{ci}) with $\nu=0$.
\begin{theorem}\label{ThTEt1}
Let $m\geq m_0=\left[\frac{n}{2}+2\right]$ and let $u$ be the unique solution on $[0,T^*[$ of  problem~(\ref{kuz})--(\ref{ci}) with $\nu=0$ for
$$u_0\in H^{m+1}(\mathbb{R}^n),\quad u_1\in H^m(\mathbb{R}^n)\quad \hbox{and}\quad \Vert u_1\Vert_{L^{\infty}(\mathbb{R}^n)}<\frac{1}{2\alpha\varepsilon}.$$
If $\sqrt{E_{m_0}[u](0)}\leq\frac{1}{4 \sqrt{B} C_{\infty}\alpha\varepsilon}$,
then
\begin{equation}\label{EqT0}
  T^*>T_0=\frac{1}{C_{m_0}\max(\alpha,\beta)\varepsilon\sqrt{B}E_{m_0}[u](0)}
\end{equation}
 and
$$u\in C^r([0,T_0];H^{m+1-r})\;\text{for}\;0\leq r\leq m+1,$$
with
$$\forall t\in[0,T_0],\;\;E_m[u](t)\leq C<+\infty.$$
Here $B$ and $C_{m_0}$ are the constants from estimate~(\ref{EqENfirstNV}) and $C_{\infty}$ is the embedding constant from the embedding of the Sobolev space $H^{[\frac{n}{2}+1]}(\mathbb{R}^n)$ in $L^{\infty}(\mathbb{R}^n)$.
\end{theorem}
\begin{proof}
 Thanks to Point~1 of Theorem~\ref{ThMainWPnu0}, for $u_0\in H^{m+1}(\mathbb{R}^n)$, $u_1\in H^m(\mathbb{R}^n)$ and $\Vert u_1\Vert_{L^{\infty}(\mathbb{R}^n)}<\frac{1}{2\alpha\varepsilon}$ there exists a unique  solution $u$ on an sufficiently small interval $[0,T]$ of problem~(\ref{kuz})--(\ref{ci}) with $\nu=0$, satisfying~(\ref{regkuznu0}) and~(\ref{conshyper}) for $s=m$. Moreover it implies that $E_m[u](0)$ is finite. Hence, we can add the hypothesis $$\sqrt{E_{m_0}[u](0)}\leq\frac{1}{4 \sqrt{B} C_{\infty}\alpha\varepsilon}$$ without adding further conditions of regularity on $u_0$ and $u_1$ as it can be reduced on a smallness condition on $\Vert u_0\Vert_{H^{m+1}}+\Vert u_1\Vert_{H^m}$.

 Let us take $T_0$, as defined in Eq.~(\ref{EqT0}), and show by induction on $j\in \mathbb{N}$ with $m_0\leq j\leq m$ that  $$\forall j\in \mathbb{N},\text{ with } m_0\leq j\leq m \quad \sup_{t\in [0,T_0]}E_j[u](t)<\infty.$$
For $j=m_0$, $u_0\in H^{m+1}(\mathbb{R}^n)\subseteq H^{m_0+1}(\mathbb{R}^n) $ and $u_1\in H^m(\mathbb{R}^n)\subseteq H^{m_0}(\mathbb{R}^n)$, and consequently $$E_{m_0}[u](0)\leq E_m [u](0) <\infty.$$

For $t\geq 0$, while  $\Vert u_t(t)\Vert_{L^{\infty}(\mathbb{R}^n)}\leq \frac{1 }{2 \alpha\varepsilon}$, it holds estimate~(\ref{EqENfirstNV}) with $m=m_0$.
According to the Gronwall Lemma, applied to~(\ref{EqENfirstNV}) with $m=m_0$, we have
$$E_{m_0}[u](t)\leq z(t),$$
where $z(t)$ is the solution of the Cauchy problem for an ordinary differential equation
$$z(t)=z_0+C_{m_0}\max(\alpha,\beta) \varepsilon \int_0^t(z(\tau))^{3/2}\dl \quad \hbox{ with } z_0=B\;E_{m_0}[u](0).$$
This problem  can be solved explicitly:
$$z(t)=\frac{z_0}{(1-\frac{1}{2}z_0^{1/2}C_{m_0} \max(\alpha,\beta)\varepsilon t)^2}.$$
We can  see that, as long as $0\leq t\leq T_0,$
the function $z(t)$ has the finite upper bound
$z(t)\leq 4 z_0.$ It implies the upper boundness of $E_{m_0}[u]$:
\begin{equation}\label{EstBEN0}
     E_{m_0}[u](t)\leq 4 B\; E_{m_0}[u](0).
  \end{equation}
Moreover, thanks to our notations,
$$\frac{\Vert u_t(t)\Vert_{L^{\infty}(\mathbb{R}^n)}}{C_{\infty}}\leq \Vert u_t(t)\Vert_{H^{[\frac{n}{2}+1]}}\leq \sqrt{ E_{m_0}[u](t)},$$
from where, using inequality~(\ref{EstBEN0}), we find
$$\Vert u_t(t)\Vert_{L^{\infty}(\mathbb{R}^n)}\leq 2  C_{\infty} \sqrt{B\; E_{m_0}[u](0)}\leq \frac{1}{2\alpha\varepsilon},$$
since $\sqrt{E_{m_0}[u](0)}\leq\frac{1}{4 \sqrt{B} C_{\infty}\alpha\varepsilon}$.
Thus Eq.~(\ref{conshyper}) holds on all interval $[0,T_0]$  and $\sup_{t\in [0,T_0]}E_{m_0}[u](t)$ is finite.

Let $j\in \mathbb{N}$, $m_0\leq j\leq m-1$ be such that $\sup_{t\in [0,T_0]}E_{j}[u](t)<\infty$.

Since Eq.~(\ref{conshyper}) holds on all interval $[0,T_0]$,  we can use the a priori estimate~(\ref{EqEnSecondNV})
and write that for all $t \in [0,T_0]$
$$E_{j+1}[u](t)\leq B\; E_{j+1}[u](0)+D_{j+1}\max(\alpha,\beta)\varepsilon \int_{0}^{t}\sqrt{E_{j}[u](\tau)}E_{j+1}[u](\tau) \dl.$$
By the induction hypothesis $\sup_{t\in [0,T_0]}E_{j}[u](t)$ is bounded by a constant, denoted here by $E^2$, and hence on $[0,T_0]$ it holds
$$E_{j+1}[u](t)\leq B E_{j+1}[u](0) +D_{j+1}\max(\alpha,\beta)\; E\varepsilon \int_0^t E_{j+1}[u](\tau) \dl.$$
Applying the Gronwall Lemma, we obtain for $t\in [0,T_0]$
$$E_{j+1}[u](t) \leq B E_{j+1}[u](0) e^{D_{j+1}\max(\alpha,\beta)\; E\varepsilon t}\leq B E_{j+1}[u](0) e^{D_{j+1}\max(\alpha,\beta)\; E\varepsilon T_0}.$$
This means, as $E_{j+1}[u](0)\leq E_m[u](0)<+\infty$, that  $\sup_{t\in [0,T_0]}E_{j+1}[u](t)<\infty$ and this finishes the proof.
\end{proof}

Theorem~\ref{ThTEt1} estimates the lifespan $T^*$ as  at least of the order $\frac{1}{\varepsilon}$, or more precisely, implies that
$$\liminf_{\varepsilon\rightarrow 0}\varepsilon T^* >0.$$
This result is independent on the dimension $n$. However, much better estimations for the lifespan can be obtained, if we use an inequality that takes into account the time decay of the solutions for $n>1$, what we do in the next section.

\subsection{Proof of Point~5 of Theorem~\ref{ThMainWPnu0}. Optimal estimations of the existence time}\label{SecJohnn4}

In~\cite{John} John uses the group of linear transformations preserving the equation $u_{tt}-\Delta u=0$. The generators of this group (the derivatives with respect to group parameters taken at the identity), here called generalized derivatives, include in addition to the derivatives $\partial_t,\partial_{x_1},\ldots,\partial_{x_n}$, first-order differential operators $L_{\alpha}$ with $\alpha=0,\ldots,n$ and $\Omega_{ik}$ with $1\leq i<k\leq n$: 
\begin{definition}\textbf{(Generalized derivatives~\cite{John})}\label{DefGenDer}
The following operators
\begin{align*}
 &L_0=t\partial_t+\sum_i x_i \partial_{x_i}, \quad L_i=x_i\partial_t+t\partial_{x_i} \hbox{  for } i=1,...,n,\\
 &\Omega_{ik}=x_i\partial_{x_k}-x_k\partial_{x_i} \hbox{ for } 1\leq i<k\leq n,\hbox{ and }\;\partial_t, \; \partial_{x_i} \hbox{ for }i=1,...,n
\end{align*}
are called the generalized derivatives.
The operators
$$L_0,\ldots,L_n,\Omega_{12},\Omega_{13},\ldots,\Omega_{n-1n},\partial_t,\partial_{x_1},\ldots,\partial_{x_n},$$
(taken in this order) are denoted respectively by $\Gamma_0,\ldots,\Gamma_{\mu}$ with $\mu=\frac{1}{2}(n^2+3n+2)$. For a multi-index $A=(A_0,\ldots,A_{\mu})$ we write in the usual way
$$\vert A\vert=A_0+\ldots+A_{\mu}, \quad \Gamma^A=(\Gamma_0)^{A_0}(\Gamma_1)^{A_1}\ldots(\Gamma_{\mu})^{A_{\mu}}.$$
\end{definition}
Therefore, in the framework of the general derivatives, we  define 
for $m\in\mathbb{N}$
\begin{align}
 &E_{\infty,m}[u](t)=\sup_x\left\vert \sup_{\vert A\vert\leq m} \left[(\Gamma^A\partial_t u(t,x))^2+ (\Gamma^A\nabla u(x,t))^2\right]\right\vert,\label{EnGenInf}\\
 &E_{1,m}[u](t)=\sum_{\vert A\vert\leq m}(\Vert \Gamma^A \partial_t u\Vert_{L^2(\mathbb{R}^n)}^2+\Vert \Gamma^A\nabla u\Vert_{L^2(\mathbb{R}^n)}^2)(t).\label{EnGen1N}
\end{align}
Let us give a remarkable estimate proved in Ref.~\cite{Klainerman2} by Klainerman:
\begin{proposition}\textbf{(Klainerman 1987)}\label{PropKlainerman}
For $n^*=[\frac{n}{2}+1]$, $m\in \mathbb{N}$, and $t>0$, as soon as $u$ is such that $E_{1,m+n^*}[u](t)$ is finite, it holds
\begin{align}\label{EstKlaire}
\sqrt{E_{\infty,m}[u](t)}\leq C_n (1+t)^{\frac{1-n}{2}}\sqrt{E_{1,m+n^*}[u](t)}.
\end{align}
\end{proposition}
Thanks to Proposition~\ref{PropKlainerman}, we improve the results of John~\cite{John} for the case of the Kuznetsov equation  and  state:
\begin{proposition}\label{Propglobapenu0}
For $n$ and $m$ in $\mathbb{N}^*$, $m\geq n+2$, let $u$ be a local solution on an interval $[0,T]$ of  problem~(\ref{kuz})--(\ref{ci}) with $\nu=0$, satisfying~(\ref{regkuznu0}) and~(\ref{conshyper}) with $s=m$. Then for all  $t\in[0,T]$, it holds
\begin{equation}\label{EstE1Nnu0}
E_{1,m}[u](t)\leq B\; E_{1,m}[u](0)+C_m\max(\alpha,\beta)\varepsilon \int_{0}^{t}(1+l)^{(1-n)/2}\;E_{1,m}[u](\tau)^{\frac{3}{2}} \dl,
\end{equation}
with a positive constant $B>0$, depending only on $c$, on $\alpha$ and on $\beta$, and with a positive constant $C_m>0$, depending only on $m$, on $n$ and on $c$.
\end{proposition}
\begin{proof} The proof follows identically the proof of Proposition~\ref{PropApEstNV1} up to Eq.~(\ref{LuDAu}) replacing everywhere $D^A$ by $\Gamma^A$.
This time Eq.~(\ref{LuDAu}) becomes
\begin{equation}
L_u\Gamma^Au= \varepsilon \sum_{j=0}^\mu \left(\alpha C_j \Gamma^{A^{j1}}u_t\;\Gamma^{A^{j2}}u_t+\sum_{i=1}^n \beta E_{ij} \Gamma^{A^{j1}}\partial_{x_i}u\;\Gamma^{A^{j2}} \partial_{x_i}u\right),
\end{equation}
where $\mu$ is defined in Definition~\ref{DefGenDer},  $C_j$ and $E_{ij}$ depend only on $\vert A\vert\leq m$, and $A^{j1}$ and $A^{j2}$ are multi-indexes, such that
$$\vert A^{j1}\vert +\vert A^{j2}\vert\leq m+1.$$
It follows that $\vert A^{j1}\vert\leq [\frac{m+1}{2}]$ or $\vert A^{j2}\vert\leq [\frac{m+1}{2}]$.
Therefore, if we set $m'=\left[\frac{m+1}{2}\right]$, we obtain
\begin{align*}
\vert J[\Gamma^Au](\tau,x)\vert\leq & C_m \max(\alpha,\beta)\varepsilon \sqrt{\sup_{\vert B\vert\leq m'} \big((\Gamma^B\partial_t u(\tau,x))^2+ (\Gamma^B\nabla u(\tau,x))^2)\big) }\cdot \\
& \cdot\sup_{\vert B\vert\leq m} \big((\Gamma^B\partial_t u(\tau,x))^2+ (\Gamma^B\nabla u(\tau,x))^2)\big)\\
\leq & C_m \max(\alpha,\beta)\varepsilon \sqrt{E_{\infty,m'}[u](\tau)} \sum_{\vert B\vert\leq m} \big((\Gamma^B\partial_t u(\tau,x))^2+ (\Gamma^B\nabla u(\tau,x))^2)\big),
\end{align*}
and thus
$$\left\vert \int_{\mathbb{R}^n}J[\Gamma^Au](\tau,x)\dx\right\vert \leq C_m\max(\alpha,\beta)\varepsilon \sqrt{E_{\infty,m'}[u](\tau)} E_{1,m}[u](\tau).$$
By hypothesis on $u$, $$\Vert u_t(t)\Vert_{L^{\infty}(\mathbb{R}^n)}\leq \frac{1}{2\alpha\varepsilon} \hbox{ on }[0,T],$$ and then, by integrating of Eq.~(\ref{relIJ})
 on $[0,t]$ with $t\in[0,T]$, we have
\begin{align*}
\frac{1}{2}\Vert\partial_t \Gamma^A u(t)\Vert_{L^2(\mathbb{R}^n)}^2+c^2\Vert \nabla \Gamma^A  u(t)\Vert_{L^2(\mathbb{R}^n)}^2\leq& \frac{3}{2} \Vert\partial_t \Gamma^A u(0)\Vert_{L^2(\mathbb{R}^n)}^2 +c^2 \Vert \nabla \Gamma^A  u(0)\Vert_{L^2(\mathbb{R}^n)}^2\\
&+ C_m \max(\alpha,\beta) \varepsilon \int_0^t \sqrt{E_{\infty,m'}[u](\tau)} E_{1,m}[u](\tau)\dl.
\end{align*}
By summing for $\vert A\vert\leq m$, we obtain
$$E_{1,m}[u](t)\leq B\; E_{1,m}[u](0)+C_m\max(\alpha,\beta)\varepsilon \int_{0}^{t}\sqrt{E_{\infty,m'}[u](\tau)} E_{1,m}[u](\tau) \dl.$$
Now we use the Klainerman  inequality~(\ref{EstKlaire}), noticing that, if we take $m\geq n+2$, we have
$$m'+n^*=\left[\frac{m+1}{2}\right]+\left[\frac{n}{2}+1\right]\leq m.$$
This finishes the proof.
\end{proof}
We  use the a priori estimate~(\ref{EstE1Nnu0})
to improve our estimation of the lifespan $T^*$ as a function of $n$.
\begin{theorem}\label{ThTimeExistNu0}
Let $m\geq n+2$. For $u_0\in H^{m+1}(\mathbb{R}^n)$ and $u_1\in H^m(\mathbb{R}^n)$ with $\Vert u_1\Vert_{L^{\infty}(\mathbb{R}^n)}\leq \frac{1}{2\alpha\varepsilon}$ we consider the local solution $u$ of problem~(\ref{kuz})--(\ref{ci}) with $\nu=0$ on an interval $[0,T]$, satisfying~(\ref{regkuznu0}) and~(\ref{conshyper}) for $s=m$ as in Point~1 of Theorem~\ref{ThMainWPnu0}. If $\sqrt{E_{1,m}[u](0)}\leq \frac{1}{4 \sqrt{B} C_{\infty}\alpha\varepsilon}$, then
$$E_{1,m}[u](t)\leq 4B\;E_{1,m}[u](0),$$ as long as
$$t\leq \Big(2C_m\max(\alpha,\beta)\varepsilon \sqrt{B\;E_{1,m}[u](0)}\Big)^{-2}\;\;\;\text{for}\;n=2,$$
$$t\leq 2\exp \Big(\frac{1}{C_m\max(\alpha,\beta)\varepsilon \sqrt{B\;E_{1,m}[u](0)}}\Big)\;\;\;\text{for}\;n=3,$$
$$1\leq\Big(2C_m\max(\alpha,\beta) \varepsilon \sqrt{B\;E_{1,m}[u](0)}\Big)^{-1} \;\;\;\text{for}\;n\geq 4.$$
Consequently,
$$\liminf_{\varepsilon\rightarrow 0} \varepsilon^2 T^* >0 \;\;\;\text{for}\;n=2,$$
$$\liminf_{\varepsilon\rightarrow 0} \varepsilon \log(T^*)>0\;\;\;\text{for}\;n=3,$$
and, for a small enough $\varepsilon$,  $T^*=+\infty \;\;\;\text{for}\; n\geq 4$, $i.e.$ the solution $u$ is  global. 
\end{theorem}
\begin{proof}
This is a direct consequence of the Gronwall lemma, used with the a priori estimate~(\ref{EstE1Nnu0}), as it is done by John in~\cite{John}.
\end{proof}
\begin{remark}
The estimations, given for $T^*$ in the case $n=1,2,3$, are optimal, as soon as, thanks to Alinhac~\cite{Alinhac},  they  give the existence time of a smooth solution of the same order as Alinhac's blow-up time, $i.e.$ up to the time of a geometrical blow-up formation.
\end{remark}

\section{Well-posedness for the viscous case}\label{secV}
\subsection{Proof of Point~1 of Theorem~\ref{ThMainWPnuPGlob}}\label{secVPr}
Let us show the global well-posedness, of the solution of the Cauchy problem~(\ref{kuz})-(\ref{ci}).
We start with the study of the  linear problem, associated to the Kuznetsov equation.
\begin{theorem}\label{ThLin1}
Let  $s\ge 0$ and $X$ be the space defined  in Point~1 of Theorem~\ref{ThMainWPnuPGlob}.
Then  the system
\begin{equation}\label{kuzlinhom}
\begin{cases}
u_{tt}-c^2\Delta u- \nu\varepsilon\Delta u_t=f, \\
u(0)=u_0,\;\;u_t(0)=u_1
\end{cases}
\end{equation}
 has a unique solution $u\in X$,
if and only if $f\in L^2(\R^+;H^s(\mathbb{R}^n)),$ $u_0\in H^{s+2}(\mathbb{R}^n)$ and $u_1\in H^{s+1}(\mathbb{R}^n)$.
Moreover it holds the following a priori estimate
\begin{equation}\label{kuzlinhomAPE}
 \Vert u\Vert_{X}\leq C \left(\Vert f \Vert_{L^2(\R^+;H^s(\mathbb{R}^n))}+\|u_0\|_{H^{s+2}(\mathbb{R}^n)}+\|u_1\|_{H^{s+1}(\mathbb{R}^n)} \right)
\end{equation}
with $\Vert u\Vert_{X}:=\Vert u\Vert_{H^2(\R^+;H^s)}+\Vert u\Vert_{L^2(\R^+;H^{s+2})}+\Vert u_t\Vert_{L^2(\R^+;H^{s+2})}$.
\end{theorem}
\begin{proof} First we take $f\in L^2(\R^+;H^s(\mathbb{R}^n)),$ $u_0\in H^{s+2}(\mathbb{R}^n)$ and $u_1\in H^{s+1}(\mathbb{R}^n)$. We use the ideas of~\cite{Haraux} (see Eq.~(4.26)).
For the sake of clarity, let us take $s=0$.
We take the inner product in $L^2(\R^n)$ of the equation with $-\Delta u_t$ and integrate by parts:
\begin{equation*}
 \frac{1}{2}\frac{d}{dt}\left(\|\nabla u_t\|^2_{L^2(\R^n)}+c^2\|\Delta u\|^2_{L^2(\R^n)} \right)+ \nu\eps \|\Delta u_t\|^2_{L^2(\R^n)}=-\int_{\R^n}f\Delta u_t\dx.
\end{equation*}
Using Young's inequality and integrating over $[0,t]$, we find
\begin{multline}\label{IneqLin}
 \frac{1}{2}\left(\|\nabla u_t\|^2_{L^2(\R^n)}+c^2\|\Delta u\|^2_{L^2(\R^n)} \right)+\frac{\nu\eps}{2} \int_0^t\|\Delta u_{\tau}\|^2_{L^2(\R^n)}\dl\\
 \le \frac{1}{2}\|\nabla u_1\|^2_{L^2(\R^n)}+\frac{1}{2}\|\Delta u_0\|^2_{L^2(\R^n)}+\frac{1}{2\nu\eps}\int_0^t\int_{\R^n}|f|^2\dx \dl.
\end{multline}
Since $f\in L^2(\R^+\times \R^n)$ and $(u_0,u_1)\in H^2(\R^n)\times H^1(\R^n)$, the last estimate implies that
$$\int_0^{+\infty}\int_{\R^n}|\Delta u_{\tau}|^2\dx\dl<+\infty.$$
Since the domain of $-\Delta$ is $H^2$, we obtain that
$$u,\; u_t\in L^2(\R^+;H^2(\mathbb{R}^n)),\quad \hbox{and} \quad u_{tt}\in L^2(\R^+\times \R^n),$$
and hence, $u\in X$ for $s=0$.
For $s>0$, as the equation is linear, we perform the same proof, using the fact that,  the operator $\Lambda= (1- \Delta)^{\frac{1}{2}}$, defined by its Fourier transform  by the formula $
\widehat{(\Lambda u)}(\zeta ) = (1+|\zeta|^2 )^{\frac{1}{2}} {\hat
u}(\zeta),$ relies  the norm of $H^s$ with the $L^2$-norm:
\begin{equation}\label{LL}
\Lambda^s= (1-\Delta)^{\frac{s}{2}}, \quad \|u\|_{H^s(\mathbb{R}^n)}=\| \Lambda^s
u \|_{L^2(\mathbb{R}^n)}.
\end{equation}
The uniqueness of $u$ follows from the linearity of the operator and the uniqueness of the solution of system (\ref{kuzlinhom}) in the case $f=0$  \cite{Ikehata}.

Conversely, if $u\in X$ solution of system (\ref{kuzlinhom}), this implies that
 $$u\in C(\R^+;H^{s+2}(\mathbb{R}^n))\quad \hbox{and} \quad u_t\in H^1(\R^+;H^{s}(\mathbb{R}^n))\cap L^2(\R^+;H^{s+2}(\mathbb{R}^n)).$$
Thanks to Theorem III.4.10.2 in~\cite{Amann}, it follows that
 $u_t\in C(\R^+;H^{s+1}(\mathbb{R}^n))$. Then we have $u(0)\in H^{s+2}(\mathbb{R}^n)$ and $u_t(0)\in H^{s+1}(\mathbb{R}^n)$. Moreover, it reads directly from the definition of $X$, that $f\in ~ L^2(\R^+;H^s(\mathbb{R}^n))$ for $u\in X$.

 The a priori estimate follows from the closed graph theorem.
\end{proof}
Let us
notice that Theorem~\ref{ThLin1} states that problem~(\ref{kuzlinhom}) has $L^2$-maximal regularity (see~\cite{Srivastava} Definition 2.1) on $\R^+$.

To be able to give a sharp estimate of the smallness of the initial data and in the same time to estimate the bound of the corresponding solution of the Kuznetsov equation (see Point~1 of Theorem~\ref{ThMainWPnuPGlob}), we use the following theorem from~\cite{Sukhinin}, which  allows us to establish our main result of the global well-posedness of the Cauchy problem for the Kuznetsov equation:
%
\begin{theorem}\label{thSuh}
(Sukhinin) Let $X$ be a Banach space, let $Y$ be a separable
topological vector space, let $L : X \rightarrow Y$ be a linear
continuous operator, let $U$ be the open unit ball in $X$, let ${\rm
P}_{LU}:LX \to [0,\infty [$ be the Minkowski functional of the set
$LU$, and let $\Phi :X \to LX$ be a mapping satisfying the condition
\begin{equation*}
 {\rm P}_{LU} \bigl(\Phi (x) -\Phi (\bar{x})\bigr) \leq
\Theta (r) \left\|x -\bar{x} \right\|\quad \text{for} \quad \left\|x
-x_0 \right\| \leqslant r,\quad \left\|\bar{x} -x_0 \right\| \leq r
\end{equation*} for some $x_0 \in X,$ where $\Theta :[0,\infty [ \to [0,\infty [$ is a monotone
non-decreasing function. Set $b(r) =\max \bigl(1 -\Theta (r),0
\bigr)$ for $r \geq 0$.

 Suppose that $$w =\int\limits_0^\infty b(r)\,dr \in ]0,\infty ], \quad r_* =\sup \{ r
\geq 0|\;b(r) >0 \},$$

$$w(r) =\int\limits_0^r b(t)dt \quad (r \geq 0) \quad\hbox{and} \quad f(x) =Lx
+\Phi(x) \quad \hbox{for} \quad x \in X.$$
Then for any $r \in
[0,r_*[$ and $ y \in f(x_0) +w(r)LU$, there exists an
 $ x \in x_0 +rU$ such that $f(x) =y$.
\end{theorem}

\begin{remark} \label{remch22.1.} If either $L$ is injective or $KerL$ has a topological
complement $E$ in $X$ such that $L(E \cap U) =LU$, then the
assertion of Theorem~\ref{thSuh} follows from the contraction
mapping principle~\cite{Sukhinin}. In particular, if $L$ is injective,
then the solution is unique.
\end{remark}

Now, we have all elements to prove Point~1 of Theorem~\ref{ThMainWPnuPGlob}:
for all $r\in [0, r^*[$ with $r^*=O(\eps^0)=O(1)$ (to be defined), as soon as the initial data are small as
\begin{equation}\label{EqSMID}
 \|u_0\|_{H^{s+2}(\mathbb{R}^n)}+\|u_1\|_{H^{s+1}(\mathbb{R}^n)}\le C\sqrt{\eps}r\quad \hbox{with } C=O(1),
\end{equation}
then the unique solution $u\in X$ satisfies
$\|u\|_X\le 2r$ ($r=O(1)$).
\begin{remark}
 It is very important to notice   that here all physical coefficients of the  Cauchy problem for the Kuznetsov equation are expressed to compare to the powers of $\eps$ ($\eps$ is the dimensionless parameter caracterising the medium perturbation as explained in~\cite{Roz3} and~\cite{Roz1}).
 In particular, if we  take into account  in Point~3 of Theorem~\ref{ThMainWPnu0}  that $c^2=O(\frac{1}{\eps})$, we  obtain the same types of smallness of the initial energy for the inviscid case as in Point~2 of Theorem~\ref{ThMainWPnuPGlob}:  $\sqrt{E_{m_0}[u](0)}\le O(\sqrt{\eps})$.
But, if we want to understand the smallness  of the initial data by their norms without the calculus of the initial energy, 
the results of Point~1 of Theorem~\ref{ThMainWPnuPGlob} can be useful.
The sharp character of Point~1 of Theorem~\ref{ThMainWPnuPGlob} can be illustrated by the following direct energy estimation approach, presented in Appendix~\ref{Appen2}.

 Let suppose that Point~2 of Theorem~\ref{ThMainWPnuPGlob} holds (see also Eq.~(\ref{EnergEN2})). Thus, for  $n\ge 3$, $m\ge \left[\frac{n}{2}+3 \right]$ if
$$\sqrt{E_{\frac{m}{2}}[u](0)}=\sqrt{\Vert \nabla u(0)\Vert_{H^m(\mathbb{R}^n)}^2+\sum_{i=1}^{\frac{m}{2}+1} \Vert \partial_t^i u(0)\Vert_{H^{m-2(i-1)}(\mathbb{R}^n)}^2}\le O(\sqrt{\eps}),$$
then it follows in a sufficient way (see Appendix~\ref{Appen2} for more details) that for $u_0\in H^{m+1}(\R^n)$ and for $u_1\in H^m(\R^n)$ it holds
\begin{equation}\label{EqPDI}
 \|\nabla u_0\|_{H^m(\mathbb{R}^n)}+\|u_1\|_{H^m(\mathbb{R}^n)}\le O(\sqrt{\eps^{m+1}}),
\end{equation}
 which implies the existence of  a unique global solution $u\in C^0(\mathbb{R}^+;H^{m+1}(\mathbb{R}^n))\cap C^1(\mathbb{R}^+;H^{m}(\mathbb{R}^n)) $ of problem~(\ref{kuz})--(\ref{ci}) such that for all $t\in \mathbb{R}^+$
$$E_{\frac{m}{2}}[u](t)\leq O\left(\frac{1}{\eps}\right)E_{\frac{m}{2}}[u](0)=O(1).$$
Thus we see that  by this approach the sufficient condition to have for all $t\geq 0$ $E_{\frac{m}{2}}[u](t)$ bounded by a constant of order zero on $\eps$ is given by Eq.~(\ref{EqPDI}) and depends on the smooth properties of the initial data (more they are regular, more they should be small). Hence, it is much more restrictive to compare to~(\ref{EqSMID}).
\end{remark}

\begin{proof}

For $u_0\in H^{s+2}(\mathbb{R}^n)$ and $u_1\in H^{s+1}(\mathbb{R}^n)$ let us denote by $u^*\in X$ the unique solution of the linear problem
\begin{equation}\label{kuzlin0inhom}
\begin{cases}
u^*_{tt}-c^2\Delta u^*-\nu\varepsilon \Delta u^*_t=0, \\
u^*(0)=u_0\in H^{s+2}(\mathbb{R}^n),\;\;u^*_t(0)=u_1\in H^{s+1}(\mathbb{R}^n).
\end{cases}
\end{equation}

In addition, according to Theorem~\ref{ThLin1}, we take $$X:=H^2(\R^+;H^{s}(\mathbb{R}^n))\cap H^1(\R^+;H^{s+2}(\mathbb{R}^n)),$$ this time for $s>\frac{n}{2}$ (we need it to control the non-linear terms),  and introduce the Banach spaces
\begin{equation}
 X_0:=\lbrace u\in X|\; u(0)=u_t(0)=0 \rbrace
\end{equation}
and $Y=L^2(\R^+;H^s(\mathbb{R}^n))$. Then by Theorem~\ref{ThLin1}, the linear operator
$$L:X_0\rightarrow Y,\quad  u\in X_0\mapsto\;L(u):=u_{tt}-c^2\Delta u-\nu \varepsilon \Delta u_t\in Y,$$
 is a bi-continuous isomorphism.

 Let us now notice that if $v$ is the unique solution of the non-linear Cauchy problem
 \begin{equation}\label{SystkuznV}
\begin{cases}
v_{tt}-c^2\Delta v-\nu\varepsilon \Delta v_t-\alpha \varepsilon (v+u^*)_t(v+u^*)_{tt}-\beta \varepsilon \nabla (v+u^*).\nabla(v+u^*)_t=0, \\
v(0)=0,\quad v_t(0)=0,
\end{cases}
\end{equation}
 then $u=v+u^*$ is the unique solution of the Cauchy problem for the Kuznetsov equation~(\ref{kuz})--(\ref{ci}).
 Let us prove the existence of a such $v$, using Theorem~\ref{thSuh}.

We suppose that $\Vert u^*\Vert_X\leq r$
and define for $v\in X_0$
$$\Phi(v):=\alpha \varepsilon (v+u^*)_t(v+u^*)_{tt}+\beta \varepsilon \nabla (v+u^*).\nabla(v+u^*)_t.$$

For $w$ and $z$ in $X_0$ such that
$\Vert w\Vert_X\leq r$ and $\Vert z\Vert_X\leq r$,
 we estimate
 \begin{multline*}
 \Vert \Phi(w)-\Phi(z)\Vert_Y=\Vert \alpha \varepsilon (u^*_t (w-z)_{tt}+(w-z)_t u^*_{tt}+w_t w_{tt}-z_t z_{tt})\\
 + \beta \varepsilon(\nabla u^* \nabla(w-z)_t+\nabla (w-z)\nabla u^*_t+\nabla w\nabla w_t-\nabla z \nabla z_t)\Vert_Y\\
 = \Vert \alpha \varepsilon (u^*_t (w-z)_{tt}+(w-z)_t u^*_{tt}+w_t (w-z)_{tt}+(w-z)_t z_{tt})\\
 + \beta \varepsilon(\nabla u^* \nabla(w-z)_t+\nabla (w-z)\nabla u^*_t+\nabla w\nabla (w-z)_t+\nabla (w-z) \nabla z_t)\Vert_Y
 \end{multline*}
by applying the triangular inequality
\begin{multline*}
\Vert \Phi(w)-\Phi(z)\Vert_Y\leq  \alpha \varepsilon \Big(\Vert u^*_t (w-z)_{tt} \Vert_Y+\Vert (w-z)_t u^*_{tt}\Vert_Y\\
+\Vert  w_t (w-z)_{tt}\Vert_Y+\Vert (w-z)_t z_{tt}\Vert_Y\Big)\\
+\beta \varepsilon\Big( \Vert \nabla u^* \nabla(w-z)_t \Vert_Y+\Vert \nabla (w-z)\nabla u^*_t \Vert_Y\\
+\Vert \nabla w\nabla (w-z)_t \Vert_Y+\Vert \nabla (w-z) \nabla z_t \Vert_Y\Big).
\end{multline*}
 Now, for all $a$ and $b$ in $X$ with $s\ge s_0> \frac{n}{2}$  it holds
 \begin{align*}
 \Vert a_t b_{tt}\Vert_Y\leq & \Vert a_t \Vert_{L^\infty(\R^+\times\mathbb{R}^n)} \Vert b_{tt}\Vert_Y\\
 \leq & C_{H^1(\R^+;H^{s_0})\to L^\infty(\R^+\times\mathbb{R}^n)} \Vert a_t\Vert_{H^1(\R^+;H^s(\mathbb{R}^n))} \Vert b\Vert_{X}\\
 \leq & C_{H^1(\R^+;H^{s_0})\to L^\infty(\R^+\times\mathbb{R}^n)} \Vert a\Vert_{X} \Vert b\Vert_{X},
\end{align*}
where $C_{H^1(\R^+;H^{s_0})\to L^\infty(\R^+\times\mathbb{R}^n)}$ is the embedding constant of $H^1(\R^+;H^{s_0})$ into the space $L^\infty(\R^+\times\mathbb{R}^n)$, independent on $s$, but depending only on the dimension $n$.
In the same way, for all $a$ and $b$ in $X$ it holds
$$\Vert \nabla a \nabla  b_t\Vert_Y\leq C_{H^1(\R^+;H^{s_0})\to L^\infty(\R^+\times\mathbb{R}^n)} \Vert a\Vert_{X} \Vert b\Vert_{X}.$$
Taking $a$ and $b$ equal to $u^*$, $w$, $z$ or $w-z$, as $\Vert u^*\Vert_X\leq r$, $\Vert w\Vert_X\leq r$ and $\Vert z\Vert_X\leq r$, we obtain
\begin{align*}
\Vert \Phi(w)-\Phi(z)\Vert_Y\leq 
4 (\alpha+\beta)C_{H^1(\R^+;H^{s_0})\to L^\infty(\R^+\times\mathbb{R}^n)} \varepsilon r \Vert w-z\Vert_X.
\end{align*}
By the fact that $L$ is a bi-continuous isomorphism, there exists a minimal constant $C_\eps=O\left(\frac{1}{\eps \nu} \right)>0$ (coming from the inequality $C_0 \eps \nu\|u\|_X^2\le \|f\|_Y\|u\|_X$ for $u$, a solution of the linear problem~(\ref{kuzlinhom}) with homogeneous initial data [for a constant $C_0=O(1)>0$ maximal])
such that
$$\forall u\in X_0 \quad \Vert u\Vert_X\leq C_\eps \Vert Lu\Vert_Y.$$
Hence, for all $f\in Y$
$$P_{LU_{X_0}}(f)\leq C_\eps P_{U_Y}(f)=C_\eps\Vert f\Vert_Y.$$
Then we find for $w$ and $z$ in $X_0$, such that $\|w\|_X\le r$, $\|z\|_X\le r$, and also with $\|u^*\|_X\le r$, that
$$P_{LU_{X_0}}(\Phi(w)-\Phi(z))\leq \Theta(r) \Vert w-z\Vert_X,$$
where $\Theta(r):= 4 C_\eps (\alpha+\beta)C_{H^1(\R^+;H^{s_0})\to L^\infty(\R^+\times\mathbb{R}^n)}\varepsilon r$.
Thus we apply Theorem~\ref{thSuh} for 
\\$f(x)=L(x)-\Phi(x)$ and $x_0=0$. Therefore, knowing that $C_\eps=\frac{C_0}{\eps \nu}$, we have, that for all  $r\in[0,r_{*}[$ with
\begin{equation}\label{Eqret}
 r_{*}=\frac{\nu}{4 C_0 (\alpha+\beta)C_{H^1(\R^+;H^{s_0})\to L^\infty(\R^+\times\mathbb{R}^n)}}=O(1),
\end{equation}
  for all $y\in \Phi(0)+w(r) L U_{X_0}\subset Y$
with $$w(r)= r-2 \frac{C_0}{\nu} C_{H^1(\R^+;H^{s_0})\to L^\infty(\R^+\times\mathbb{R}^n)} (\alpha+\beta) r^2,$$
there exists a unique $v\in 0+r U_{X_0}$ such that $L(v)-\Phi(v)=y$.
But, if we want that $v$ be the solution of the non-linear Cauchy problem~(\ref{SystkuznV}), then we need to impose $y=0$, and thus to ensure that $0\in \Phi(0)+w(r) L U_{X_0}$.
Since $-\frac{1}{w(r)}\Phi(0)$ is an element of $Y$ and $LX_0=Y$, there exists a unique $z\in X_0$ such that
\begin{equation}\label{Eqz}
 L z=-\frac{1}{w(r)}\Phi(0).
\end{equation}
Let us show that $\|z\|_X\le 1$, what will implies that $0\in \Phi(0)+w(r) L U_{X_0}$.
Noticing that
\begin{align*}
\Vert \Phi(0)\Vert_Y & \leq \alpha \varepsilon \Vert v_t v_{tt}\Vert_Y +\beta \varepsilon \Vert \nabla v \nabla v_t\Vert_Y\\
& \leq  (\alpha+\beta) \varepsilon C_{H^1(\R^+;H^{s_0})\to L^\infty(\R^+\times\mathbb{R}^n)}\Vert v\Vert_X^2 \\
& \leq (\alpha+\beta) \varepsilon C_{H^1(\R^+;H^{s_0})\to L^\infty(\R^+\times\mathbb{R}^n)}r^2
\end{align*}
and using~(\ref{Eqz}), we find
\begin{align*}
 & \Vert z\Vert_X \leq C_\eps\Vert L z\Vert_Y=C_\eps\frac{\Vert \Phi(0)\Vert_Y}{w(r)}\\
 &\leq \frac{C_\eps C_{H^1(\R^+;H^{s_0})\to L^\infty(\R^+\times\mathbb{R}^n)} (\alpha+\beta) \varepsilon r}{(1-2 C_\eps C_{H^1(\R^+;H^{s_0})\to L^\infty(\R^+\times\mathbb{R}^n)} (\alpha+\beta)\varepsilon r)}<\frac{1}{2},
\end{align*}
as soon as $r<r^*$.

Consequently, $z\in U_{X_0}$ and $\Phi(0)+w(r) Lz=0$.

Then we conclude that  for all  $r\in[0,r_{*}[$, if $\|u^*\|_X\le r$, there exists a unique $v\in r U_{X_0}$ such that $L(v)-\Phi(v)=0$, $i.e.$  the solution of the non-linear Cauchy problem~(\ref{SystkuznV}).
Thanks to the maximal regularity and a priori estimate following from inequality~(\ref{IneqLin}) with $f=0$,
there exists a constant $C_1=O(\eps^0)>0$, such that
$$\|u^*\|_X\le \frac{C_1}{\sqrt{\nu \eps}}(\Vert u_0\Vert_{H^{s+2}(\mathbb{R}^n)}+\Vert u_1\Vert_{H^{s+1}(\mathbb{R}^n)}).$$

Thus, for all  $r\in[0,r_{*}[$ and $\Vert u_0\Vert_{H^{s+2}(\mathbb{R}^n)}+\Vert u_1\Vert_{H^{s+1}(\mathbb{R}^n)}\le \frac{\sqrt{\nu \eps}}{C_1}r$, the function $u=u^*+v\in X$ is the unique solution of the Cauchy problem for the Kuznetsov equation and $\Vert u\Vert_X\leq 2 r$.
\end{proof}

\subsection{Proof of Point~2 of Theorem~\ref{ThMainWPnuPGlob}: Case $n\ge 3$}\label{sec4}

 Knowing the existence of a solution $u$ of the Kuznetsov equation in $$X=H^2(\R^+;H^{s}(\mathbb{R}^n))\cap H^1(\R^+;H^{s+2}(\mathbb{R}^n)),$$ we notice that this directly implies that
 $$u\in C(\R^+;H^{s+2}(\mathbb{R}^n))\quad \hbox{and} \quad u_t\in H^1(\R^+;H^{s}(\mathbb{R}^n))\cap L^2(\R^+;H^{s+2}(\mathbb{R}^n)).$$ By Theorem III.4.10.2 in~\cite{Amann}, it implies that
 $u_t\in C(\R^+;H^{s+1}(\mathbb{R}^n))$, which gives that $$u\in C^1(\R^+;H^{s+1}(\mathbb{R}^n))\cap C(\R^+;H^{s+2}(\mathbb{R}^n))$$ and, this time with the help of the Kuznetsov equation, $u_{tt}\in C(\R^+;H^{s-1}(\mathbb{R}^n))$.
 Consequently, in the viscous case the regularity of the time derivatives of the order greater than two of the solutions differs  from the regularity, obtained in Section~\ref{sec2} for the  inviscid case. Thus we have to consider estimates with different energies:
the energy $E_{\frac{m}{2}}[u](t)$, defined in Eq.~(\ref{EnergEN2}), and the energy

\begin{equation}\label{EnSm2}
 S_{\frac{m}{2}}[u](t)=\sum_{i=1}^{\frac{m}{2}+1} \Vert\nabla  \partial_t^i u(t)\Vert_{H^{m-2(i-1)}(\mathbb{R}^n)}^2,
\end{equation}
 defined, as $E_{\frac{m}{2}}[u](t)$,  for $m\in\mathbb{N}$ and $m$ even, which respect to the obtained regularity of $u$ and its derivatives.

\begin{lemma}\label{PropEnergiVisc}
Let $n\in\mathbb{N}^*$, $n\geq 3$, $m\in \mathbb{N}$, and  $u$ be the solution of  problem~(\ref{kuz})-(\ref{ci}). Then  for $m\geq \left[\frac{n}{2}+3\right]$, $m$ even,  and all multi-index $A=(A_0,A_1,...,A_n)$ with $\vert A\vert-A_0 \leq m-2A_0$ it holds
\begin{equation}\label{apglobalindA}
\begin{aligned}
\frac{d}{dt}\Big(\int_{\mathbb{R}^n} ((1-\alpha\varepsilon u_t) & (D^A u_t)^2+c^2 (\nabla D^A u)^2))(\tau,x)\;\dx\Big) \\
&+2\nu\varepsilon \int_{\mathbb{R}^n} (\nabla D^A u_t)^2(\tau,x)\;\dx\\
& \leq C_m \max(\alpha,\beta)\varepsilon \sqrt{E_{\frac{m}{2}}[u](\tau)}S_{\frac{m}{2}}[u](\tau)
\end{aligned}
\end{equation}
with a constant $C_m>0$, depending only on $m$ and on the dimension $n$.
\end{lemma}
\begin{proof}
Following notations of the proof of Proposition~\ref{PropApEstNV1} in Annexe~\ref{Appen1}, we redefine
$$L_u v:= v_{tt}-c^2\Delta v-\nu\varepsilon \Delta v_t-\alpha\varepsilon u_t\;  v_{tt}-\beta \varepsilon \nabla u \;\nabla v_t,$$
 where $u$ is the solution of  problem~(\ref{kuz}). For this new $L_u v$ with the additional term $\nu\varepsilon \Delta v_t$,
we have a modified version of  relation~(\ref{relIJ})
\begin{equation}\label{relIJN}
 \frac{d}{dt}\int_{\mathbb{R}^n}I[v](t,x)\dx+2 \nu\varepsilon \int_{\mathbb{R}^n} (\nabla v_t)^2\dx=\int_{\mathbb{R}^n}J[v](t,x)\dx,
\end{equation}
where $I[v]$ and $J[v]$ are defined in Eqs.~(\ref{eqI})--(\ref{eqJ}).
We still take $v=D^A u$ with $A=(A_0,A_1,...,A_n)$, but this time $\vert A\vert-A_0 \leq m-2A_0$ and $m$ is even. Then we just need to show
\begin{equation}\label{EqEstJnu}
 \left\vert \int_{\mathbb{R}^n}J[D^A u](t,x)\dx\right\vert\leq \varepsilon C_m \max(\alpha,\beta) \sqrt{E_{\frac{m}{2}}[u](t)}S_{\frac{m}{2}}[u](t).
\end{equation}

For $n\geq 3$, $m\geq \left[\frac{n}{2}+3\right]$ and $m$ even, we have, thanks to the H\"{o}lder inequality,
\begin{align*}
\int_{\mathbb{R}^n} \vert u_{tt} (D^Au_t)^2\vert\dx \leq & \Vert u_{tt}\Vert_{L^{\frac{n}{2}}(\mathbb{R}^n)} \Vert D^A u_t\Vert_{L^{\frac{2n}{n-2}}(\mathbb{R}^n)}^2.
\end{align*}
Noticing, that, thanks to Ref.~\cite{ADAMS-1975} Theorem~7.57 p.~228,
for $s>\frac{n}{2}$ there hold the continuous embeddings $H^s(\R^n)\subset C_B^0(\R^n)\subset L^{\frac{n}{2}}(\R^n)$ (where $C_B^0$ is the Banach space of bounded continuous functions equal to zero at the infinity), we can write for $m\ge \left[\frac{n}{2}+3\right]$
\begin{equation}\label{EstuttLn2}
 \Vert u_{tt}\Vert_{L^{\frac{n}{2}}(\R^n)}\leq C\|u_{tt}\|_{H^{[\frac{n}{2}+1]}(\R^n)}\leq C\|u_{tt}\|_{H^{m-2}(\R^n)}\le C\sqrt{E_{\frac{m}{2}}[u]}.
\end{equation}

In addition, with the help of the Gagliardo-Nirenberg-Sobolev inequality
\begin{equation}\label{EqGNS}
 \|v\|_{L^{\frac{2n}{n-2}}(\R^n)}\leq C \|\nabla v\|_{L^2(\R^n)},
\end{equation}
we also have
$$\Vert D^A u_t\Vert_{L^{\frac{2n}{n-2}}(\mathbb{R}^n)}\le C \Vert \nabla D^A u_t\Vert_{L^{2}(\mathbb{R}^n)}\le C\Vert \nabla D^{A_0+1}_t u\Vert_{H^{|A|-A_0}(\mathbb{R}^n)}.$$
With the hypothesis that $\vert A\vert-A_0 \leq m-2A_0$, there hold $2A_0\le m$ and
$$\Vert \nabla D^{A_0+1}_t u\Vert_{H^{|A|-A_0}(\mathbb{R}^n)}\le \Vert \nabla D^{A_0+1}_t u\Vert_{H^{m-2A_0}(\mathbb{R}^n)}.$$
Therefore, all norms $\Vert \nabla D^{A_0+1}_t u\Vert^2_{H^{m-2A_0}(\mathbb{R}^n)}$, for the chosen $n,$ $m$ and $A_0$, are present in $S_{\frac{m}{2}}$. Hence, we find
\begin{align}\label{EqEstuttnu}
\int_{\mathbb{R}^n} \vert u_{tt} (D^Au_t)^2\vert\dx
\leq  C \Vert u_{tt}\Vert_{H^{m-2}(\mathbb{R}^n)} \Vert\nabla D^A u_t\Vert_{L^2(\mathbb{R}^n)}^2
\leq C \sqrt{E_{\frac{m}{2}}[u]} S_{\frac{m}{2}}[u],
\end{align}
and in the same way,
\begin{align*}
\int_{\mathbb{R}^n} \vert \Delta u (D^Au_t)^2\vert\dx \leq & \Vert\Delta u\Vert_{L^{\frac{n}{2}}(\mathbb{R}^n)} \Vert D^Au_t\Vert_{L^{\frac{2n}{n-2}}(\mathbb{R}^n)}^2\leq  C \Vert \Delta u \Vert_{H^{[\frac{n}{2}+1]}(\mathbb{R}^n)} \Vert\nabla D^A u_t\Vert_{L^2(\mathbb{R}^n)}^2\\
\leq & C \sqrt{E_{\frac{m}{2}}[u]} S_{\frac{m}{2}}[u].
\end{align*}

To calculate $L_u D^A u$ we use  expression~(\ref{LuDAu}) with multi-indexes $A^{j1}$ and $A^{j2}$ satisfying~(\ref{propind}). As in the proof of Proposition~\ref{PropApEstNV1}, without loss of generality, we consider two multi-indexes $A^1$ and $A^2$ with the same properties~(\ref{propind}).
We perform two steps:
\begin{description}
\item[Step 1] we prove \begin{equation}\label{EstStep1nu}
                        \int_{\mathbb{R}^n}\vert D^{A^1}u_t\; D^{A^2}u_t\;D^Au_t\vert\dx\leq C \sqrt{E_{\frac{m}{2}}[u]} S_{\frac{m}{2}}[u],
                       \end{equation}

\item[Step 2] we prove \begin{equation}\label{EstStep2nu}\int_{\mathbb{R}^n}\vert D^{A^1}\partial_{x_i}u\; D^{A^2}\partial_{x_i}u\;D^Au_t\vert\dx\leq C \sqrt{E_{\frac{m}{2}}[u]} S_{\frac{m}{2}}[u].
\end{equation}
\end{description}

\vspace*{4pt}\noindent\textbf{Step 1.}
Thanks to  properties~(\ref{propind}) of $A_1$ and $A_2$ and to the symmetry of the general case
$$\int_{\R^n} |(D^{A_0^1}_t D_x^{(A_1^1,\ldots,A_n^1)} u_t) (D^{A_0^2}_t D_x^{(A_1^2,\ldots,A_n^2)} u_t) (D^Au_t)| \dx,$$
we divide our proof on three typical cases:
\begin{description}
\item[Case 1]$\vert A^1\vert-A^1_0\geq 0$, $A^1_0 \geq 0$, $\vert A^2\vert-A^2_0>0$ and $A^2_0>0$, $i.e.$ a non trivial presence of $D_t^{A_0^2}$ and $D_x^{(A_1^2,\ldots,A_n^2)}$ is imposed, 
\item[Case 2]$\vert A^1\vert-A^1_0= 0$, $A^1_0 > 0$, $\vert A^2\vert-A^2_0 > 0$ and $A^2_0=0$, $i.e.$ we consider the integrals of the form $\int_{\R^n} |(D^{A_0^1}_t  u_t) (D_x^{(A_1^2,\ldots,A_n^2)} u_t) (D^Au_t)| \dx,$
\item[Case 3]$\vert A^1\vert-A^1_0= 0$, $A^1_0 > 0$, $\vert A^2\vert-A^2_0=0$ and $A^2_0>0$, $i.e.$ we consider only non-trivial time derivatives $\int_{\R^n} |(D^{A_0^1}_t  u_t) (D^{A_0^2}_t  u_t) (D^Au_t)| \dx.$
\end{description}

\vspace*{4pt}\noindent\textbf{Step 1, Case 1.} By the generalized H\"{o}lder inequality with $\frac{1}{p}+\frac{1}{q}=\frac{n+2}{2n}$, we have
\begin{align*}
\int_{\mathbb{R}^n}\vert D^{A^1}u_t\; D^{A^2}u_t\;D^Au_t\vert\dx\leq & \Vert D^{A^1}u_t\Vert_{L^p(\mathbb{R}^n)}\Vert D^{A^2}u_t\Vert_{L^q(\mathbb{R}^n)}\Vert D^A u_t\Vert_{L^{\frac{2n}{n-2}}(\mathbb{R}^n)}.
\end{align*}
By the Sobolev embeddings~(\ref{injsoblp}) of $H^{m_1}\subset L^p$ and $H^{m_2}\subset L^q$ with $m_1+m_2=\frac{n}{2}-1$ and $0<m_1<\frac{n}{2}-1$, we find
\begin{align*}
\int_{\mathbb{R}^n}\vert D^{A^1}u_t\; D^{A^2}u_t\;D^Au_t\vert\dx\leq& C \Vert D^{A^1}u_t\Vert_{H^{m_1}(\mathbb{R}^n)}\Vert D^{A^2}u_t\Vert_{H^{m_2}(\mathbb{R}^n)}\Vert\nabla D^A u_t\Vert_{L^{2}(\mathbb{R}^n)},
\end{align*}
where we have also applied the Gagliardo-Nirenberg-Sobolev inequality~(\ref{EqGNS}).
Hence,
\begin{align}
\int_{\mathbb{R}^n}\vert D^{A^1}u_t & \; D^{A^2}u_t\;D^Au_t\vert\dx\nonumber\\
\leq& C\Vert \partial_t^{A^1_0} u_t\Vert_{H^{m_1+\vert A^1\vert-A^1_0}(\mathbb{R}^n)}\Vert\nabla \partial_t^{A^2_0}u_t\Vert_{H^{m_2+\vert A^2\vert-A^2_0-1}(\mathbb{R}^n)}  S_{\frac{m}{2}}[u]^{\frac{1}{2}}.\label{Ineq2et}
\end{align}
Now we are looking for $0<m_1<\frac{n}{2}-1$, such that
\begin{equation}\label{Systm1A2}
\begin{cases}
m_1+\vert A^1\vert-A^1_0\leq m-2A^1_0,\\
m_2+\vert A^2\vert-A^2_0-1\leq m-2A^2_0,
\end{cases}
\end{equation}
in order to have
\begin{equation}\label{EqToHave1}
 \Vert \partial_t^{A^1_0} u_t\Vert_{H^{m_1+\vert A^1\vert-A^1_0}(\mathbb{R}^n)}\leq \sqrt{E_{\frac{m}{2}}[u]}\quad \hbox{and}\quad \Vert\nabla \partial_t^{A^2_0}u_t\Vert_{H^{m_2+\vert A^2\vert-A^2_0-1}(\mathbb{R}^n)}\leq\sqrt{S_{\frac{m}{2}}[u]}.
\end{equation}
Since $m_2=\frac{n}{2}-1-m_1$, and by~(\ref{propind}), $|A^2|=|A|+1-|A^1|$ and $A^2_0=A_0+1-A^1_0$, system~(\ref{Systm1A2}) is equivalent to\begin{equation*}
 \begin{cases}
m_1+\vert A^1\vert+A^1_0\leq m,\\
\frac{n}{2}-1-m_1+\vert A\vert+1-\vert A^1\vert+A_0+1-A^1_0-1\leq m.
\end{cases}
\end{equation*}
The last system, thanks to  $\vert A\vert+A_0\leq m$, corresponding to the assumptions of the Proposition, is satisfied if 
$$\frac{n}{2}\leq m_1+\vert A^1\vert+A^1_0\leq m.$$
Using~(\ref{propind}), we find that $$\vert A^1\vert+A^1_0=|A|+A_0+2-(|A^2|+A_0^2).$$
Therefore, since for Case~1 $|A^2|\ge 2$ and $A_0^2\ge 1$, recalling that (again by~(\ref{propind})) $|A|+A_0\le m$, we obtain
$$ 1\leq\vert A^1\vert+A^1_0\le m-1.$$

Thus, we distinguish three sub-cases:

\begin{description}
\item[For $n\ge 3$,  $\frac{n}{2} \leq\vert A^1\vert+A^1_0 \leq m-1$] taking $m_1=\frac{1}{4}$, we obtain~(\ref{EqToHave1}).
\item[For $n\ge 5$,  $2 \leq \vert A^1\vert+A^1_0 < \frac{n}{2}$] as $m\ge \left[\frac{n}{2}+3\right]$, it is sufficient to take $m_1= \frac{n}{2}- (\vert A^1\vert+A^1_0 )$.
  \item[For $n\ge 3$,  $\vert A^1\vert+A^1_0=1$] instead of finding $m_1$, we notice, that we have only two possibility: either $D^{A^1}=\partial_t$  and $A^2=A$, which gives  estimate~(\ref{EqEstuttnu}), or $D^{A^1}=\partial_{x_i} $ with $A^2_0=A_0+1$ and $\vert A^2\vert-A^2_0=\vert A\vert-A_0-1>0$. For the last case,  by the generalized H\"{o}lder inequality, we have
\begin{align}\label{EqEstuxinu}
\int_{\mathbb{R}^n}\vert \partial_{x_i} u_t\;D^{A^2}u_t\;D^Au_t\vert dx\leq \Vert \partial_{x_i} u_t\Vert_{L^n(\mathbb{R}^n)} \Vert D^{A^2}u_t\Vert_{L^2(\mathbb{R}^n)} \Vert D^Au_t\Vert_{L^{\frac{2n}{n-2}}(\mathbb{R}^n)}.
\end{align}
For $m\geq \left[\frac{n}{2}+3\right]$ the first norm in Eq.~(\ref{EqEstuxinu}) can be estimated using the continuous embedding $H^s(\mathbb{R}^n)\subset L^n (\R^n)$ holding for $s>\frac{n}{2}$: 
$$\Vert \partial_{x_i} u_t\Vert_{L^n(\mathbb{R}^n)}\leq C\Vert \partial_{x_i} u_t\Vert_{H^{[\frac{n}{2}+1]}(\mathbb{R}^n)}\leq C \Vert u_t\Vert_{H^{m-1}(\mathbb{R}^n)}\leq C\sqrt{E_{\frac{m}{2}}[u]}  . $$
With the help of the Gagliardo-Nirenberg-Sobolev inequality (\ref{EqGNS}), we also estimate the second norm in~(\ref{EqEstuxinu})
\begin{equation}\label{EstDAutLdrob}
 \Vert D^Au_t\Vert_{L^{\frac{2n}{n-2}}(\mathbb{R}^n)}\leq C \Vert \nabla D^Au_t\Vert_{L^2(\mathbb{R}^n)}\leq C \sqrt{S_{\frac{m}{2}}[u]} ,
\end{equation}
and for the last one we directly have
$$\Vert D^{A^2}u_t\Vert_{L^2(\mathbb{R}^n)}\leq \Vert \nabla \partial_t^{A_0+2} u\Vert_{H^{\vert A\vert-A_0-2}(\mathbb{R}^n)}\leq \Vert \nabla \partial_t^{A_0+2} u\Vert_{H^{m-2 A_0-2}(\mathbb{R}^n)}\leq \sqrt{S_{\frac{m}{2}}[u]}.$$
Thus we obtain as previously estimate~(\ref{EstStep1nu}) of Step~1.
  \end{description}
This permits to  conclude Case 1 of Step~1.

\vspace*{4pt}\noindent\textbf{Step 1, Case 2.} We have $\vert A^1\vert-A^1_0= 0$, $A^1_0 > 0$, $\vert A^2\vert-A^2_0 > 0$ and $A^2_0=0$. Therefore, by~(\ref{propind}), $A_0^1=1+A_0$, and, updating~(\ref{Ineq2et}), we directly have
\begin{align*}
\int_{\mathbb{R}^n}\vert D^{A^1_0}_tu_t\; D^{(A^2_1,\ldots ,A^2_n)}_x u_t\;D^Au_t\vert\dx\leq & C\Vert \partial_t^{A_0+1} u_t\Vert_{H^{m_1}(\mathbb{R}^n)}\Vert\nabla u_t\Vert_{H^{m_2+\vert A^2\vert-1}(\mathbb{R}^n)}  S_{\frac{m}{2}}[u]^{\frac{1}{2}}
\end{align*}
with $m_1+m_2=\frac{n}{2}-1$, $0<m_1<\frac{n}{2}-1$. Now we need to find $m_1$, belonging to $]0,\frac{n}{2}-1[$, such that
\begin{equation}\label{Sysm1S1C2nu}
\begin{cases}
m_1 \leq m-2(A_0+1),\\
m_2+\vert A^2\vert-1 \leq m,
\end{cases}
\end{equation}
in order to have
$$\Vert \partial_t^{A_0+1} u_t\Vert_{H^{m_1}(\mathbb{R}^n)}\leq \sqrt{E_{\frac{m}{2}}[u]}\quad \hbox{and} \quad\Vert\nabla u_t\Vert_{H^{m_2+\vert A^2\vert-1}(\mathbb{R}^n)}\leq\sqrt{S_{\frac{m}{2}}[u]}.$$
From  $1+|A|=|A^1|+|A^2|$, by~(\ref{propind}), with the relation $|A^1|=A_0^1=1+A_0$ it follows that
\begin{equation}\label{EqindA2C12}
  \vert A^2\vert=\vert A\vert-A_0 .
          \end{equation}
 Therefore, as $m_2=\frac{n}{2}-m_1-1$, system~(\ref{Sysm1S1C2nu}) is equivalent to
$$
\begin{cases}
m_1+2 A_0 \leq m-2,\\
\frac{n}{2}-2\leq m_1+m-\vert A\vert+A_0.
\end{cases}
$$
By the assumption of the proposition
\begin{equation}\label{EqindmC12}
       m-\vert A\vert+A_0\geq 2A_0,
      \end{equation}
 hence the last system is satisfied if we have $m_1$ such that
$$\frac{n}{2}-2\leq m_1 +2A_0 \leq m-2.$$
Knowing that  $\vert A^2\vert>0$ (by the assumption of Case 2), Eq.~(\ref{EqindA2C12}) implies that
$\vert A\vert-A_0>0$. Thus,  relation~(\ref{EqindmC12}) gives $2A_0\leq m-1$, or more precisely $$2A_0\leq m-2,$$ since $m$ is even.
So, a $m_1$ with $0<m_1<\frac{n}{2}-1$ exists if $m-2A_0>2$.
Indeed, if $2A_0<\frac{n}{2}-2$
we can take $m_1=\frac{n}{2}-2-2A_0$,
and if $m-3\geq 2A_0\geq\frac{n}{2}-2$ we can take $m_1=\frac{1}{2}$.

Let us now consider the limit case $2A_0=m-2$. Then we have $\vert A^1\vert=A^1_0=\frac{m}{2}$. Moreover, from~(\ref{EqindmC12}) viewed, thanks to Eq.~(\ref{EqindA2C12}), as $|A^2|+2A_0\le m$, follows that $1\leq \vert A^2\vert\leq 2$. We apply the generalized H\"{o}lder inequality and estimate~(\ref{EqGNS}) to obtain
\begin{align*}
\int_{\mathbb{R}^n}\vert \partial_{t}^{\frac{m}{2}} u_t\;D^{(A^2_1,\ldots,A^2_n)}_xu_t\;D^Au_t\vert dx\leq & \Vert \partial_{t}^{\frac{m}{2}} u_t\Vert_{L^2(\mathbb{R}^n)} \Vert D^{(A^2_1,\ldots,A^2_n)}_x u_t\Vert_{L^n(\mathbb{R}^n)} \Vert D^Au_t\Vert_{L^{\frac{2n}{n-2}}(\mathbb{R}^n)}\\
\leq & C \Vert \partial_{t}^{\frac{m}{2}} u_t\Vert_{L^2(\mathbb{R}^n)} \Vert D^{(A^2_1,\ldots,A^2_n)}_xu_t\Vert_{L^n(\mathbb{R}^n)} \sqrt{S_{\frac{m}{2}}[u]}.
\end{align*}
 Moreover,
$$\Vert \partial_{t}^{\frac{m}{2}} u_t\Vert_{L^2(\mathbb{R}^n)}\leq \sqrt{E_{\frac{m}{2}}[u]}.$$
Using the  continuity of the embedding $H^s(\mathbb{R}^n)\subset L^n (\R^n)$ for $s>\frac{n}{2}$, we also find  for $m\geq \left[\frac{n}{2}+3\right]$
\begin{align*}
\Vert D^{A^2}u_t\Vert_{L^n(\mathbb{R}^n)}\leq & C\Vert D^{A^2}u_t\Vert_{H^{[\frac{n}{2}+1]}(\mathbb{R}^n)}\leq C \Vert \nabla u_t\Vert_{H^{[\frac{n}{2}+2]}(\mathbb{R}^n)}\\
\leq& C\Vert \nabla u_t\Vert_{H^{m}(\mathbb{R}^n)} \leq C \sqrt{S_{\frac{m}{2}}[u]}.
\end{align*}
Hence, estimate~(\ref{EstStep1nu}) of Step~1 is also proved for Case~2.

\vspace*{4pt}\noindent\textbf{Step 1, Case 3.}
Let us notice that thanks to relations~(\ref{propind}), from $|A^1|=A_0^1$ and $|A^2|=A_0^2$ it follows $|A|=A_0$.
We start as usual with the generalized H\"{o}lder inequality
\begin{align*}
\int_{\mathbb{R}^n}\vert D^{A^1_0}_tu_t\; D^{A^2_0}_tu_t\;D^{A_0}_tu_t\vert\dx\leq & \Vert D^{A^1_0}_tu_t\Vert_{L^p(\mathbb{R}^n)}\Vert D^{A^2_0}_tu_t\Vert_{L^q(\mathbb{R}^n)}\Vert D^{A_0}_tu_t\Vert_{L^{\frac{2n}{n-2}}(\mathbb{R}^n)}
\end{align*}
with $\frac{1}{p}+\frac{1}{q}=\frac{n+2}{2n}$. Then we apply the Gagliardo-Nirenberg-Sobolev inequality~(\ref{EqGNS}) and its more general version, which can be viewed as the embedding of the Sobolev space $W^1_{q^*}(\R^n)$ in the Lebesgue space $L^q(\R^n)$ with $\frac{1}{q}=\frac{1}{q^*}- \frac{1}{n}$ and $1\leq q^*<n$:
\begin{align*}
\int_{\mathbb{R}^n}\vert D^{A^1_0}_tu_t\; D^{A^2_0}_tu_t\;D^{A_0}_t u_t\vert\dx\leq & C \Vert D^{A^1_0}_tu_t\Vert_{L^p(\mathbb{R}^n)}\Vert \nabla D^{A^2_0}_tu_t\Vert_{L^{q^*}(\mathbb{R}^n)}\Vert\nabla D^{A_0}_tu_t\Vert_{L^{2}(\mathbb{R}^n)}
\end{align*}
with  $\frac{1}{p}+\frac{1}{q^*}=\frac{n+4}{2n}$.
We notice that if we want to use the Sobolev embeddings~(\ref{injsoblp}) to $L^p$ and to $L^{q^*}$, it is only possible if $\frac{1}{p}$ and $\frac{1}{q^*}$ are smaller then $\frac{1}{2}$, or equivalently, if $\frac{1}{p}+\frac{1}{q^*}=\frac{n+4}{2n}<1$.  Knowing that  $\frac{n+4}{2n}<1$ for  $n\ge 5$,  $\frac{n+4}{2n}> 1$ for $n=3$ and $\frac{n+4}{2n}= 1$ for $n=4$, we  treat   separately two cases: $n\geq 5$ and $n=3$ or $4$.

For $n=3$ or $4$, we choose
 $p=\frac{n}{2}$ and $q=\frac{2n}{n-2}$, implying $q^*=2$. Thus, for $n=3$ we use the continuous embedding
 $H^2(\R^3)\subset L^{\frac{3}{2}}(\R^3)$~\cite{ADAMS-1975} (since $2>\frac{3}{2}$) and for $n=4$ we use $H^2(\R^4)\subset L^{2}(\R^4)$ to obtain
\begin{align*}
\int_{\mathbb{R}^n}\vert D^{A^1_0}_tu_t\; D^{A^2_0}_tu_t\;D^{A_0}_t u_t\vert\dx\leq & \Vert D^{A^1_0}_tu_t\Vert_{L^{\frac{n}{2}}(\mathbb{R}^n)}\Vert\nabla D^{A^2_0}_tu_t\Vert_{L^{2}(\mathbb{R}^n)}\Vert\nabla D^{A_0}_t u_t\Vert_{L^{2}(\mathbb{R}^n)}\\
\leq & C \Vert D^{A^1_0}_tu_t\Vert_{H^2(\mathbb{R}^n)} S_{\frac{m}{2}}[u].
\end{align*}
If $m-2A_0^1\ge 2$, then we directly  have
$$\Vert D^{A^1_0}u_t\Vert_{H^2(\mathbb{R}^n)}\leq \Vert D^{A^1_0}u_t\Vert_{H^{m-2A_0^1}(\mathbb{R}^n)}\leq \sqrt{E_{\frac{m}{2}}[u]}.$$
Recalling that $m$ is even, and, by  our assumption $|A^1|+A_0^1\le m$, $2A_0^1\le m$, there is  only one additional possibility: $m-2A_0^1=0$, $i.e.$ $A_0^1=\frac{m}{2}$.

For $A^1_0=\frac{m}{2}$, thanks to~(\ref{propind}) and the assumption $2A_0\le m$, we necessary have $\vert A^2_0\vert=1$,   and consequently, by~(\ref{EstDAutLdrob}),
\begin{align*}
\int_{\mathbb{R}^n}\vert \partial_t^{\frac{m}{2}}u_t\; u_{tt}\;\partial_t^{\frac{m}{2}}u_t\vert\dx
 \leq C \Vert u_{tt}\Vert_{H^2(\mathbb{R}^n)} \Vert \partial_t^{\frac{m}{2}}u_t\Vert_{L^{\frac{2n}{n-2}}(\mathbb{R}^n)}^2
\leq  \sqrt{E_{\frac{m}{2}}[u]}S_{\frac{m}{2}}[u].
\end{align*}
Thus for $n=3$ and $n=4$ we find estimate~(\ref{EstStep1nu}).

Now, for $n\geq 5$, when $\frac{1}{p}+\frac{1}{q^*}=\frac{n+4}{2n}<1$,   we have
\begin{align*}
\int_{\mathbb{R}^n}\vert D^{A^1_0}_tu_t\; D^{A^2_0}_tu_t\;D^{A_0}_tu_t\vert\dx\leq & C \Vert D^{A^1_0}_tu_t\Vert_{L^p(\mathbb{R}^n)}\Vert \nabla D^{A^2_0}_tu_t\Vert_{L^{q^*}(\mathbb{R}^n)}\Vert\nabla D^{A_0}_t u_t\Vert_{L^{2}(\mathbb{R}^n)}\\
\leq & C \Vert D^{A^1_0}_tu_t\Vert_{H^{m_1}(\mathbb{R}^n)} \Vert\nabla D^{A^2_0}_tu_t\Vert_{H^{m_2}(\mathbb{R}^n)} \sqrt{S_{\frac{m}{2}}[u]}
\end{align*}
with $m_1+m_2=\frac{n}{2}-2$ and $0<m_1<\frac{n}{2}-2$ by the Sobolev embeddings~(\ref{injsoblp}) which give us $H^{m_1}\subset L^p$ and $H^{m_2}\subset L^{q^*}$.
We look for $m_1$ such that
\begin{equation}\label{Eqsys}
 m_1 \leq m-2A^1_0, \quad m_2 \leq m-2A^2_0
\end{equation}
in order to have
$$\Vert D^{A^1_0}_t u_t\Vert_{H^{m_1}}\leq \sqrt{E_{\frac{m}{2}}[u]}\quad \hbox{and} \quad\Vert\nabla D^{A^2_0}_tu_t\Vert_{H^{m_2}}\leq\sqrt{S_{\frac{m}{2}}[u]}.$$
As $m_2=\frac{n}{2}-2-m_1$ and $ A^2_0= A_0+1-A^1_0$, system~(\ref{Eqsys})  is equivalent to
$$
\begin{cases}
m_1 +2A^1_0 \leq m,\\
\frac{n}{2}-2 \leq m-2A_0+m_1+2A^1_0-2.
\end{cases}
$$
As $m-2 A_0\geq 0$, it is sufficient to have $m_1$ such that
$$\frac{n}{2}\leq m_1+2 A^1_0\leq m$$
with $0<m_1<\frac{n}{2}-2$ and $1\leq A^1_0\leq \frac{m}{2}$. 
If $2\leq A^1_0<\frac{n}{4}$ we can take $m_1= \frac{n}{2}-2 A^1_0$. And if $\frac{n}{4}\leq A^1_0\leq\frac{m}{2}-1$ we can take $m_1=\frac{1}{4}$.

If $A^1_0=1$, then necessary $A^2_0=A_0$, and using estimates~(\ref{EstuttLn2}) and~(\ref{EstDAutLdrob}) we directly find
\begin{align*}
\int_{\mathbb{R}^n}\vert u_{tt}\; (D^{A_0}_tu_t)^2\vert\dx\leq & C \Vert u_{tt}\Vert_{L^{\frac{n}{2}}(\mathbb{R}^n)}\Vert  D^{A^2_0}_tu_t\Vert_{L^{\frac{2n}{n-2}}(\mathbb{R}^n)}^2\leq C \sqrt{E_{\frac{m}{2}}[u]} S_{\frac{m}{2}}[u].
\end{align*}

If  $A^1_0=\frac{m}{2}$ we are in a symmetric case as $A^2_0=1$.
This conclude the proof of Case 3 and of  Step 1, $i.e.$ of estimate~(\ref{EstStep1nu}).

\vspace*{4pt}\noindent\textbf{Step 2.} Let us show estimate~(\ref{EstStep2nu}).
Thanks to  properties~(\ref{propind}) of $A_1$ and $A_2$ and to the symmetry of the general case
$$\int_{\R^n} |(D^{A_0^1}_t D_x^{(A_1^1,\ldots,A_n^1)} u_{x_i}) (D^{A_0^2}_t D_x^{(A_1^2,\ldots,A_n^2)} u_{x_i}) (D^Au_t)| \dx,$$
we divide our proof on two typical cases:

\begin{description}
\item[Case 1]$\vert A^1\vert-A^1_0\geq 0$, $A^1_0 > 0$, $\vert A^2\vert-A^2_0 \geq 0$ and $A^2_0>0$, $i.e.$
a non trivial presence of $D^{A_0^1}_t$ and $D^{A_0^2}_t$ is imposed,
\item[Case 2]$\vert A^1\vert-A^1_0> 0$, $A^1_0 = 0$, $\vert A^2\vert-A^2_0 \geq 0$ and $A^2_0> 0$, $i.e.$
we consider the integrals of the form $\int_{\R^n} |( D_x^{A_1^1+\ldots+A_n^1} u_{x_i}) (D^{A_0^2}_t D_x^{A_1^2+\ldots+A_n^2} u_{x_i}) (D^Au_t)| \dx$ with a non-trivial $D^{A_0^2}_t$.
\end{description}

\vspace*{4pt}\noindent\textbf{Case 1.} Using estimate $\Vert D^Au_t\Vert_{L^2}\le \sqrt{E_{\frac{m}{2}}[u]}$, we have
\begin{align*}
&\int_{\R^n} |(D^{A_0^1}_t D_x^{(A_1^1,\ldots,A_n^1)} u_{x_i}) (D^{A_0^2}_t D_x^{(A_1^2,\ldots,A_n^2)} u_{x_i}) (D^Au_t)| \dx \\
&\leq C \Vert\nabla \partial_t^{A^1_0} u\Vert_{H^{m_1+\vert A^1\vert-A^1_0}(\mathbb{R}^n)} \Vert \nabla \partial_t^{A^2_0} u\Vert_{H^{m_2+\vert A^2\vert-A^2_0}(\mathbb{R}^n)} \sqrt{E_{\frac{m}{2}}[u]}
\end{align*}
with $m_1+m_2=\frac{n}{2}$ and $0<m_1<\frac{n}{2}$.

Let us find $m_1$ with $0<m_1<\frac{n}{2}$ such that
\begin{equation}\label{systmS2v}
 \begin{cases}
m_1+\vert A^1\vert-A^1_0 \leq m-2(A^1_0-1),\\
m_2+\vert A^2\vert-A^2_0 \leq  N-2(A^2_0-1)
\end{cases}
\end{equation}
in order to have
$$\Vert\nabla \partial_t^{A^1_0} u\Vert_{H^{m_1+\vert A^1\vert-A^1_0}(\mathbb{R}^n)}\leq \sqrt{S_{\frac{m}{2}}[u]}\quad \hbox{and} \quad\Vert \nabla \partial_t^{A^2_0} u\Vert_{H^{m_2+\vert A^2\vert-A^2_0}(\mathbb{R}^n)}\leq\sqrt{S_{\frac{m}{2}}[u]}.$$
As $m_2=\frac{n}{2}-m_1$, $\vert A^1\vert+\vert A^2\vert=\vert A\vert+1$, and $A^1_0+A^2_0=A_0+1$, system~(\ref{systmS2v}) is equivalent to $$
\begin{cases}
m_1+\vert A^1\vert+A^1_0 \leq m+2,\\
\frac{n}{2}+\vert A\vert+A_0 +2\leq  m+2 +m_1+\vert A^1\vert+A^1_0.
\end{cases}
$$
By our assumption $\vert A\vert +A_0\leq m $, and hence the last system  is satisfied if  $m_1$ verifies
$$\frac{n}{2}\leq m_1+\vert A^1\vert+A^1_0\leq m+2.$$
In our case $A_0^1>0$, thus $2\leq \vert A^1\vert+A^1_0\leq m$, which implies the existence of a such $m_1$ with $0<m_1<\frac{n}{2}$. Indeed, if $m\geq \vert A^1\vert+A^1_0\geq \frac{n}{2}$ we can take $m_1=1$, else if $2\leq \vert A^1\vert+A^1_0<\frac{n}{2}$ it is possible to take $m_1=\frac{n}{2}-(\vert A^1\vert+A^1_0)$. This concludes Case 1 of Step 2.

\vspace*{4pt}\noindent\textbf{Case 2.}
Thanks to~(\ref{propind}), the conditions $|A^1|>0$ with $A_0^1=0$ imply that $\vert A\vert-A_0>0$. Consequently, with $m_1+m_2=\frac{n}{2}$ and $0<m_1<\frac{n}{2}$ as in the previous case, we obtain
\begin{align*}
\int_{\mathbb{R}^n}\vert D^{A^1}_x\partial_{x_i} u & \; D^{A^2}\partial_{x_i} u\;D^Au_t\vert\dx\\
\leq & C \Vert \nabla u\Vert_{H^{m_1+\vert A^1\vert}(\mathbb{R}^n)} \Vert \nabla \partial_t^{A^2_0} u\Vert_{H^{m_2+\vert A^2\vert-A^2_0}(\mathbb{R}^n)}\Vert \nabla \partial_t^{A_0} u_t\Vert_{H^{\vert A\vert -A_0-1}(\mathbb{R}^n)}\\
\leq&  C \Vert \nabla u\Vert_{H^{m_1+\vert A^1\vert}(\mathbb{R}^n)} \Vert \nabla \partial_t^{A^2_0} u\Vert_{H^{m_2+\vert A^2\vert-A^2_0}(\mathbb{R}^n)} \sqrt{S_{\frac{m}{2}}[u]}.
\end{align*}
In the aim to have
$$\Vert \nabla u\Vert_{H^{m_1+\vert A^1\vert}(\mathbb{R}^n)}\leq \sqrt{E_{\frac{m}{2}}[u]}\quad \hbox{and} \quad\Vert \nabla \partial_t^{A^2_0} u\Vert_{H^{m_2+\vert A^2\vert-A^2_0}(\mathbb{R}^n)}\leq\sqrt{S_{\frac{m}{2}}[u]},$$
we need to find $m_1$ with $0<m_1<\frac{n}{2}$, such that
$$
\begin{cases}
m_1+\vert A^1\vert \leq m,\\
m_2+\vert A^2\vert-A^2_0 \leq  m-2(A^2_0-1).
\end{cases}
$$
As $m_2=\frac{n}{2}-m_1$, $\vert A^2\vert=\vert A\vert +1-\vert A^1\vert$ and $A^2_0=A_0+1$ it is equivalent to solve
$$
\begin{cases}
m_1+\vert A^1\vert \leq m,\\
\frac{n}{2}-m_1+\vert A\vert+1-\vert A^1\vert +A_0+1-2 \leq  m.
\end{cases}
$$
As $m-\vert A\vert-A_0\geq 0$, the last system is satisfied if $m_1$ verifies
$$\frac{n}{2}\leq m_1+\vert A^1\vert\leq m. $$
By assumptions of this case it tolds $1\leq \vert A^1\vert\leq m$, what guarantees the existence of such $m_1$ with $0<m_1<\frac{n}{2}$. Indeed, if $1\leq \vert A^1\vert < \frac{n}{2}$, then we can take $m_1= \frac{n}{2}- \vert A^1\vert$, and if  $\frac{n}{2}\leq\vert A^1\vert\leq m-1$, then we can take $m_1=\frac{1}{2}$.
In the case $\vert A^1\vert=m$, corresponding to $D^{A^2}=\partial_t$, we directly obtain
\begin{align*}
\int_{\mathbb{R}^n}&\vert D^{A^1}_x\partial_{x_i} u\;\partial_{x_i} u_t\;D^Au_t\vert\dx
\leq  C\Vert D^{A^1}_x\partial_{x_i} u\Vert_{L^2(\mathbb{R}^n)} \Vert \partial_{x_i} u_t\Vert_{L^n(\mathbb{R}^n)} \Vert D^A u_t\Vert_{L^{\frac{2n}{n-2}}(\mathbb{R}^n)}\\
&\leq  C \Vert \nabla u\Vert_{H^m} \Vert \nabla u_t\Vert_{H^m(\mathbb{R}^n)} \Vert\nabla D^A u_t\Vert_{L^{2}(\mathbb{R}^n)}
\leq  C \sqrt{E_{\frac{m}{2}}[u]}S_{\frac{m}{2}}[u].
\end{align*}
This completes the proof of Step 2 and hence the proof of estimate~(\ref{EstStep2nu}).

Thus, estimates~(\ref{EstStep1nu}) and~(\ref{EstStep2nu}) imply
$$\left\vert \int_{\mathbb{R}^n} L_u D^Au D^Au_t\dx\right\vert  \leq C \max(\alpha,\beta) \varepsilon \sqrt{E_{\frac{m}{2}}[u]}S_{\frac{m}{2}}[u],$$
from where follows~(\ref{EqEstJnu}).
\end{proof}

Thanks to Lemma~\ref{PropEnergiVisc}, we have the following energy decreasing result: 
\begin{theorem}\label{ThDecrEnnu}
Let $n\geq 3$, $m\in\mathbb{N}$ be even and $m\geq \left[\frac{n}{2}+3\right]$. For $u_0\in H^{m+1}(\mathbb{R}^n)$ and $u_1\in H^m(\mathbb{R}^n)$, satisfying the smallness condition according to Point~1 of Theorem~\ref{ThMainWPnuPGlob}, there exists a unique global solution 
$$u\in C^1(\R^+;H^{m-1}(\mathbb{R}^n))\cap C(\R^+;H^{m}\break(\mathbb{R}^n))$$ 
of  problem~(\ref{kuz})--(\ref{ci}) and the energy $E_{\frac{m}{2}}[u](0)<\infty$ is well-defined. Then
\begin{enumerate}
 \item it holds the a priori estimate
 \begin{equation}\label{EstAPKuzn}
  \frac{d}{dt}E(t)+\sqrt{2}\varepsilon S_{\frac{m}{2}}[u](t)\left(\sqrt{2}\nu-  C_m \max(\alpha,\beta) \sqrt{E(t)}\right)\le 0,
 \end{equation}
 where, denoting by $V$ the set of all multi-indexes $A=(A_0,A_1,...,A_n)$ with $\vert A\vert-A_0 \leq m-2A_0$,
$$E(t)=\sum_{A\in V} \int_{\mathbb{R}^n} (1-\alpha\varepsilon u_t) (D^A u_t)^2+c^2 (\nabla D^A u)^2)(t,x)\;\dx.$$
\item if in addition $\sqrt{E_{\frac{m}{2}}[u](0)}\leq \frac{\sqrt{2}\nu}{ \sqrt{\frac{3}{2}+c^2}C_m \max(\alpha,\beta)  }=O(\sqrt{\eps}),$ then
\begin{equation}\label{EqEm2Bound}
 \forall t\in \mathbb{R}^+,\;\;E_{\frac{m}{2}}[u](t)\leq (3+2c^2)E_{\frac{m}{2}}[u](0)=O(1).
\end{equation}
\end{enumerate}
\end{theorem}
\begin{proof}
We sum~(\ref{apglobalindA}) on all  $A\in V$ to obtain
 $$\frac{d}{dt}E(t)+2\nu \varepsilon S_{\frac{m}{2}}[u]\leq C_m \max(\alpha,\beta) \varepsilon \sqrt{E_{\frac{m}{2}}[u]}S_{\frac{m}{2}}[u].$$
 While $\Vert u_t(t)\Vert_{L^{\infty}(\mathbb{R}^n)}\leq \frac{1}{2\alpha\varepsilon}$ it holds
 $$\frac{1}{2} E_{\frac{m}{2}}[u](t)\leq E(t) \leq (\frac{3}{2}+c^2)E_{\frac{m}{2}}[u](t),$$
and consequently,
 $$ \frac{d}{dt}E(t)+2\nu\varepsilon S_{\frac{m}{2}}[u](t)\leq \sqrt{2} C_m \max(\alpha,\beta) \varepsilon \sqrt{E(t)} S_{\frac{m}{2}}[u](t).$$
 Thus, if for all time $\sqrt{E(t)}< \frac{\sqrt{2}\nu}{\max(\alpha,\beta) C_m  }$, and in particular,
 \begin{equation}\label{EqE0petit}
   E(0)\leq \left(\frac{3}{2}+c^2\right) E_{\frac{m}{2}}[u](0)<2 \left(\frac{\nu}{ C_m \max(\alpha,\beta)  }\right)^2,
 \end{equation}
  then we have the decreasing of $E$ in time:
  $$\frac{d}{dt}E(t)<0\quad \hbox{and} \quad  E(t)\leq E(0).$$
 Moreover,  for all time $t\ge 0$
 \begin{align*}
 \Vert u_t(t)\Vert_{L^{\infty}(\mathbb{R}^n)}\leq& C_{\infty} \sqrt{E_{\frac{m}{2}}[u](t)}
 \leq  C_{\infty} \sqrt{2}\sqrt{E(t)}
 \leq  C_{\infty} \sqrt{2}\sqrt{E(0)}\\
 < & 2C_{\infty} \frac{\nu}{C_m \max(\alpha,\beta)}
 < \frac{1}{2\alpha\varepsilon}.
 \end{align*}
 To be able to write $2C_{\infty} \frac{\nu}{C_m \max(\alpha,\beta)}
 < \frac{1}{2\alpha\varepsilon}$, we recall that, using the physical values of coefficients, $\eps\ll 1$, $c^2=O(\frac{1}{\eps})$, $\alpha=\frac{\gamma-1}{c^2}<\beta=2$, and consequently, as $\nu=O(1)$, the last inequality becomes
 $$\frac{C_\infty}{C_m}\nu<\frac{1}{2\alpha \eps},$$
 which is obviously true in the case of $\eps\ll 1$ (and, for instance, taking $C_m=2C_\infty$).
 Hence, if Eq.~(\ref{EqE0petit}) holds, then  for all time $\Vert u_t(t)\Vert_{L^{\infty}}< \frac{1}{2\alpha\varepsilon}$ and the well-posedness of the Cauchy problem is ensured with the following energy estimate
 $$E_{\frac{m}{2}}[u](t)\leq 2 E(0)\leq (3+2c^2)E_{\frac{m}{2}}[u](0).$$
 \end{proof}

\appendix
\section{Proof of Proposition~\ref{PropApEstNV1}}\label{Appen1}

 Following~\cite{John}, let us consider
\begin{equation}\label{EqLuv}
 L_u v= v_{tt}-c^2\Delta v-\alpha\varepsilon u_t  v_{tt}-\beta \varepsilon \nabla u \;\nabla v_t,
\end{equation}
where $u$ is a local solution on $[0,T]$ of  problem~(\ref{kuz})--(\ref{ci}) with $\nu=0$, satisfying~(\ref{regkuznu0}) and~(\ref{conshyper}) for $s=m$.
We multiply  Eq.~(\ref{EqLuv}) by $v_t$ and integrate over $\mathbb{R}^n$\begin{align*}
&\int_{\mathbb{R}^n}L_u v\;v_t\dx\\=&\frac{1}{2}\frac{d}{dt}\left(\int_{\mathbb{R}^n} v_t^2+c^2(\nabla v)^2\dx\right)-\alpha\varepsilon \int_{\mathbb{R}^n} u_t v_{tt} v_t\dx-\beta\varepsilon\int_{\mathbb{R}^n} \nabla u \nabla v_t v_t\dx\\
=&\frac{1}{2}\frac{d}{dt}\left(\int_{\mathbb{R}^n} v_t^2+c^2(\nabla v)^2\dx\right) -\frac{\alpha}{2}\varepsilon\left[\frac{d}{dt} \left(\int_{\mathbb{R}^n} u_t\; v_t^2\dx\right)-\int_{\mathbb{R}^n} u_{tt}\; v_t^2\dx\right] \\
&+\frac{\beta}{2}\varepsilon \int_{\mathbb{R}^n} \Delta u \; (v_t)^2\dx.
\end{align*}
Hence, denoting by
\begin{align}
I[v]= v_t^2+c^2(\nabla v)^2-\alpha\varepsilon  u_t\; v_t^2,\label{eqI}\\
J[v]=2 L_u v\;v_t-\left[ \alpha\varepsilon u_{tt}+\beta\varepsilon \Delta u\right](v_t)^2,\label{eqJ}
\end{align}
we have the following equation
\begin{equation}\label{relIJ}
 \frac{d}{dt}\int_{\mathbb{R}^n}I[v](t,x)\dx=\int_{\mathbb{R}^n}J[v](t,x)\dx.
\end{equation}
Let $A=(A_0,A_1,...,A_n)$ be a multi-index, and $D^Av=\partial_t^{A_0}\partial_{x_1}^{A_1}...\partial_{x_n}^{A_n}$. To prove estimate~(\ref{EqENfirstNV}), we study $\vert \int_{\mathbb{R}^n}J[v](t,x)\dx\vert$ for $v=D^Au$ with $\vert A\vert=A_0+...+A_n\leq m$.

For $m\geq \left[\frac{n}{2}+2\right]$ and a multi-index $A$ with $\vert A\vert\leq m$ we estimate, thanks to the definition of $E_m[u]$ by Eq.~(\ref{EnergEN}),
\begin{align}
\int_{\mathbb{R}^n} \vert u_{tt} (D^Au_t)^2\vert\dx \leq & \Vert u_{tt}\Vert_{L^{\infty}(\mathbb{R}^n)} \Vert D^Au_t\Vert_{L^2(\mathbb{R}^n)}^2\nonumber\\
\leq & C \Vert u_{tt}\Vert_{H^{[\frac{n}{2}+1]}(\mathbb{R}^n)} E_{m}[u]
\leq  C\;E_m[u]^{\frac{3}{2}},\label{Est1a}
\end{align}
with a constant $C>0$, depending only on $n$ by the Sobolev embedding~\cite{ADAMS-1975}~Theorem~ 7.57~p.~228
\begin{equation}\label{injsoblinf}
H^s(\mathbb{R}^n)\hookrightarrow L^{\infty}(\mathbb{R}^n)\;\;\;\;\;\;\text{for }s>\frac{n}{2}.
\end{equation}
In the same way, using the Sobolev embedding~(\ref{injsoblinf}), we obtain
\begin{align}
\int_{\mathbb{R}^n} \vert \Delta u (D^Au_t)^2\vert\dx\leq & \Vert\Delta u\Vert_{L^{\infty}(\mathbb{R}^n)} \Vert D^Au_t\Vert_{L^2(\mathbb{R}^n)}^2 \leq  C \Vert \Delta u \Vert_{H^{[\frac{n}{2}+1]}(\mathbb{R}^n)}E_{m}[u]\nonumber\\
\leq & C \Vert\nabla u\Vert_{H^{m}(\mathbb{R}^n)} E_{m}[u]
\leq  C\; E_m[u]^{\frac{3}{2}}.\label{Est1b}
\end{align}

To calculate $L_u D^A u$ we apply the chain rule of differentiation to $D^AL_u u=0$. As $L_uu=0$ we suppose $\vert A\vert\geq1$.
By developing $D^A(\nabla u \nabla u_t)=\sum_{i=1}^n D^A(\partial_{x_i}u \partial_{x_i}u_t)$ with 
$D^A(u_t \;u_{tt})$, we have
\begin{equation}\label{LuDAu}
L_uD^Au= \varepsilon \sum_j \left(C_j\alpha D^{A^{j1}}u_t\;D^{A^{j2}}u_t+\sum_{i=1}^n E_{ij}\beta D^{A^{j1}}\partial_{x_i}u\;D^{A^{j2}} \partial_{x_i}u\right),
\end{equation}
where $\sum_j$ is a finite sum, with $C_j$ and $E_{ij}$ depending only on $\vert A\vert\leq m$, and $A^{j1}$ and $A^{j2}$ are multi-index such that
\begin{equation}\label{propind}
\left\lbrace
\begin{array}{l}
\vert A^{j1}\vert+\vert A^{j2}\vert=\vert A\vert+1,\\
\vert A^{j1}\vert\geq 1,\;\;\vert A^{j2}\vert\geq 1,\\
A^{j1}_0+ A^{j2}_0=A_0+1,\;\;A^{j1}_i+ A^{j2}_i=A_i\;\;\text{for}\;1\leq i\leq n.\\
\end{array}\right.
\end{equation}

Let us show for $m\geq \left[\frac{n}{2}+2\right]$ the estimate
\begin{equation}\label{EstForProofNV}
 \left| \int_{\mathbb{R}^n} L_u D^A u\; D^A u_t\;\dx\right|\leq C \varepsilon \max(\alpha,\beta) E_m[u]^{\frac{3}{2}}.
\end{equation}
Without loss of generality, we consider two multi-indexes $A^1$ and $A^2$ satisfying~(\ref{propind}) and divide the proof of~(\ref{EstForProofNV}) in two parts: we estimate 
$\int_{\mathbb{R}^n}\vert D^{A^{1}}u_t\;D^{A^{2}}u_t \;D^A u_t\vert\dx$ first, and secondary   $\int_{\mathbb{R}^n}\vert D^{A^{1}}\partial_{x_i} u\;D^{A^{2}}\partial_{x_i}u \;D^A  u_t\vert\dx$. As the proof of each part is very similar, we give the details only for the first one.

 To estimate  $\int_{\mathbb{R}^n}\vert D^{A^{1}}u_t\;D^{A^{2}}u_t \;D^A u_t\vert\dx$, we consider three cases:
 \begin{description}
  \item[Case 1] $1<\vert A^1\vert<m$ and $1<\vert A^2\vert<m$,
  \item[Case 2] $\vert A^1\vert\leq m$ and $\vert A^2\vert=1$,
  \item[Case 3] $\vert A^2\vert\leq m$ and $\vert A^1\vert=1$.
 \end{description}
 Let us detail
\textbf{Case 1} (other cases can be treated in a similar way).

For $2\le\vert A^1\vert\le m-1$ and $2\le\vert A^2\vert\le m-1$, it holds\begin{align*}
\int_{\mathbb{R}^n}\vert D^{A^{1}}u_t\;D^{A^{2}}u_t \;D^A u_t\vert\dx\leq & \Vert D^{A^{1}}u_t \Vert_{L^p(\mathbb{R}^n)} \Vert D^{A^{2}}u_t\Vert_{L^q(\mathbb{R}^n)}\Vert D^A u_t \Vert_{L^2(\mathbb{R}^n)},
\end{align*}
with $\frac{1}{p}+\frac{1}{q}=\frac{1}{2}$ by the general H\"{o}lder inequality~\cite{BREZIS-2005}.
Hence, using the Sobolev embedding~\cite{ADAMS-1975}
\begin{equation}\label{injsoblp}
H^{m_1}(\mathbb{R}^n)\hookrightarrow L^p(\mathbb{R}^n)\;\;\; \text{with }\frac{1}{p}=\frac{1}{2}-\frac{m_1}{n}\text{ and }0 < m_1 < \frac{n}{2},\end{equation}
we find
\begin{align*}
&\int_{\mathbb{R}^n}\vert D^{A^{1}}u_t\;D^{A^{2}}u_t \;D^A u_t\vert\dx \leq C \Vert D^{A^{1}}u_t \Vert_{H^{m_1}(\mathbb{R}^n)} \Vert D^{A^{2}}u_t\Vert_{H^{\frac{n}{2}-m_1}(\mathbb{R}^n)} \Vert D^A u_t \Vert_{L^2(\mathbb{R}^n)}.
\end{align*}
In what follows by $C>0$ is denoted an arbitrary constant depending only on $m$ and on $n$.

We have
$$\Vert D^{A^{1}}u_t \Vert_{H^{m_1}(\mathbb{R}^n)}\leq \Vert \partial_t^{A^{1}_0}u_t \Vert_{H^{m_1+\vert A^1\vert-A^1_0}(\mathbb{R}^n)},$$ 
$$ \Vert D^{A^{2}}u_t\Vert_{H^{\frac{n}{2}-m_1}(\mathbb{R}^n)}\leq \Vert \partial_t^{A^{2}_0}u_t\Vert_{H^{\frac{n}{2}-m_1+\vert A^2\vert-A^2_0}(\mathbb{R}^n)}.$$
We need to find $m$ for which there exists $m_1$ with $0 < m_1 < \frac{n}{2}$, such that
\begin{equation}\label{sysind}
\left\lbrace\begin{array}{rl}
m_1+\vert A^1\vert-A^1_0 & \leq m+1-(A^1_0+1),\\
\frac{n}{2}-m_1+\vert A^2\vert-A^2_0 & \leq m+1-(A^2_0+1),\\
\end{array}\right.
\end{equation}
or equivalently, by~(\ref{propind}) $|A^2|=|A|+1-|A^1|$,
$$\left\lbrace\begin{array}{rl}
m_1+\vert A^1\vert & \leq m,\\
\frac{n}{2}-m_1+\vert A\vert+1-\vert A^1\vert & \leq m.\\
\end{array}\right.$$
As $m-\vert A\vert\geq 0$ it is sufficient to find $m_1$, such that
$$\frac{n}{2}+1\leq m_1+\vert A^1\vert \leq m$$
with $2\leq \vert A^1\vert \leq m-1$ and $0<m_1<\frac{n}{2}$. In particular, the last three inequalities imply that $m\ge [2+\frac{n}{2}]$. For the existence of $m_1$, we  see that, for instance,
\begin{description}
 \item[if $\vert A_1\vert=2$] we can  take $m_1=\frac{n}{2}-\frac{1}{4}$,
 \item[if $2<\vert A_1\vert<\frac{n}{2}+1$] we can take $m_1=\frac{n}{2}+1-\vert A_1\vert$,
 \item[if $\frac{n}{2}+1\leq \vert A_1\vert\leq m-1$] we can take $m_1=\frac{1}{4}$.
\end{description}

Moreover,
$$\Vert D^Au_t\Vert_{L^2(\mathbb{R}^n)}\leq \Vert \partial_t^{A_0}u_t\Vert_{H^{\vert A\vert-A_0}(\mathbb{R}^n)}\leq \Vert \partial_t^{A_0}u_t\Vert_{H^{m-A_0}(\mathbb{R}^n)}.$$
Then, thanks to relations~(\ref{sysind}), we conclude
\begin{align*}
\int_{\mathbb{R}^n}\vert D^{A^{1}}u_t\;D^{A^{2}}u_t \;D^A u_t\vert\dx \leq & C\Vert \partial_t^{A^{1}_0}u_t\Vert_{H^{m-A^{1}_0}(\mathbb{R}^n)} \Vert \partial_t^{A^{2}_0}u_t\Vert_{H^{m-A^{2}_0}(\mathbb{R}^n)} \Vert \partial_t^{A_0}u_t\Vert_{H^{m-A_0}(\mathbb{R}^n)}\\
\leq &C \;E_m[u]^{\frac{3}{2}}.
\end{align*}

Consequently, for $m\geq \left[\frac{n}{2}+2\right]$, and $A^1$ and $A^2$, satisfying  properties~(\ref{propind}), it holds
\begin{equation}\label{ap1kuz0}
\int_{\mathbb{R}^n}\vert D^{A^{1}}u_t\;D^{A^{2}}u_t \;D^A u_t\vert\dx\leq C\; E_m[u]^{\frac{3}{2}}.
\end{equation}

By the same argument, for $m\geq \left[\frac{n}{2}+2\right]$ and $A^1$ and $A^2$, satisfying  properties~(\ref{propind}), we  control   the terms of the form $\int_{\mathbb{R}^n}\vert D^{A^{1}}\partial_{x_i} u\;D^{A^{2}}\partial_{x_i}u \;D^A  u_t\vert\dx$:
\begin{equation}\label{ap2kuz0}
\int_{\mathbb{R}^n}\vert D^{A^{1}}\partial_{x_i} u\;D^{A^{2}}\partial_{x_i}u \;D^A  u_t\vert\dx\leq C E_m[u]^{\frac{3}{2}}.
\end{equation}
Thus, considering~(\ref{LuDAu}),~(\ref{ap1kuz0}) and~(\ref{ap2kuz0}) for $m\geq \left[\frac{n}{2}+2\right]$ and for a multi-index $A$ with $\vert A\vert\leq m$, we have estimate~(\ref{EstForProofNV}).

 Thanks to estimates~(\ref{Est1a}),~(\ref{Est1b}) and~(\ref{EstForProofNV}), we are able to control each term of $J[D^Au]$ from Eq.~(\ref{eqJ}):
\begin{equation}\label{EstJ}
\left\vert \int_{\mathbb{R}^n}J[D^Au](t,x)\dx\right\vert \leq C \max(\alpha,\beta) \varepsilon E_m[u](t)^{\frac{3}{2}}.
\end{equation}
 By the hypothesis  that  $u$ is a local solution of the inviscid Kuznetsov equation, $u$ satisfies Eq.~(\ref{conshyper}), $i.e.$  $\Vert u_t(t)\Vert_{L^{\infty}}\leq \frac{1}{2 \alpha\varepsilon}$ on $[0,T]$, which implies the equivalence  of energies
$$\int_{\mathbb{R}^n} \frac{1}{2} (D^Au_t)^2+c^2(\nabla D^Au)^2\dx \leq \int_{\mathbb{R}^n} I[D^Au]\dx\leq  \int_{\mathbb{R}^n} \frac{3}{2} (D^Au_t)^2+c^2(\nabla D^Au)^2\dx.$$
 We integrate relation~(\ref{relIJ}) over $[0,t]$ with $t\leq T$ to obtain
\begin{align*}
\Vert D^Au_t(t)\Vert_{L^2(\mathbb{R}^n)}^2+& \Vert \nabla D^Au(t)\Vert_{L^2(\mathbb{R}^n)}^2\\
\leq &\frac{(\frac{3}{2}+c^2)}{\min(1/2,c^2)} (\Vert D^Au_t(0)\Vert_{L^2(\mathbb{R}^n)}^2+ \Vert \nabla D^Au(0)\Vert_{L^2(\mathbb{R}^n)v}^2)\\
&+\frac{1}{\min(1/2,c^2)}\int_0^t\int_{\mathbb{R}^n}J(\tau,x)\dx\;\dl.
\end{align*}
Then, using estimate~(\ref{EstJ}), we find
\begin{align*}
\Vert D^Au_t(t)\Vert_{L^2(\mathbb{R}^n)}^2+ & \Vert \nabla D^Au(t)\Vert_{L^2(\mathbb{R}^n)}^2\\
\leq &\frac{(\frac{3}{2}+c^2)}{\min(1/2,c^2)} (\Vert D^Au_t(0)\Vert_{L^2(\mathbb{R}^n)}^2+ \Vert \nabla D^Au(0)\Vert_{L^2(\mathbb{R}^n)}^2)\\
&+\frac{1}{\min(1/2,c^2)}C\max(\alpha,\beta)\varepsilon\int_0^t E_m[u](\tau)^{\frac{3}{2}}\dl.
\end{align*}
As we have this for all multi-index $A$ with $\vert A\vert\leq m$, by summing, we obtain
$$E_m[u](t)\leq\frac{(3+2c^2)}{\min(1/2,c^2)} \; E_m[u](0)+\frac{C\max(\alpha,\beta)}{\min(1/2,c^2)}\varepsilon \int_0^t E_m[u](\tau)^{\frac{3}{2}} \dl$$
with a constant $C>0$, depending only of $n$ and $m$. This gives estimate~(\ref{EqENfirstNV}).
\begin{remark}\label{RemAprop2}
  To prove estimate~(\ref{EqEnSecondNV}) it is sufficient to show, using the  proof of Proposition~\ref{PropApEstNV1}, that
 for $m\geq \left[\frac{n}{2}+3\right]$ and all multi-index $A$ with $\vert A\vert\leq m$
$$\left\vert \int_{\mathbb{R}^n}J[D^Au](s,x)\dx \right\vert \leq C \varepsilon \sqrt{E_{m-1}[u]}E_m[u],$$
where
$J[D^Au]$ is defined in Eq.~(\ref{eqJ}).
 \end{remark}

\section{Illustration of the sharp behavior of Point~1 in Theorem~\ref{ThMainWPnuPGlob}}\label{Appen2}
\begin{theorem}
Let $n\geq 3$, $m\in \mathbb{N}$ be even, $m\geq[\frac{n}{2}+3]$. For $u_0\in H^{m+1}(\mathbb{R}^n)$ and $u_1\in H^m(\mathbb{R}^n)$
 if
\begin{align}\label{controlinitdata}
\Vert \nabla u_0\Vert_{H^m(\mathbb{R}^n)}+& \Vert u_1\Vert_{H^m(\mathbb{R}^n)} \nonumber\\
\leq& \sqrt{\frac{1}{1+\frac{(2c^2+2)^{m+2}-1}{(2c^2+2)^2-1}}\frac{2\nu^2}{(\frac{3}{2}+c^2)C_m^2 \max(\alpha^2,\beta^2)}}=O(\sqrt{\eps^{m+1}}),
\end{align}
then $\sqrt{E_{\frac{m}{2}}[u](0)}\leq \frac{\sqrt{2}\nu}{ \sqrt{\frac{3}{2}+c^2}C_m \max(\alpha,\beta)  }=O(\sqrt{\eps}),$
so that by Theorem \ref{ThDecrEnnu} Point $2$ there exists a unique global solution $u\in C^0(\mathbb{R}^+;H^{m+1}(\mathbb{R}^n))\cap C^1(\mathbb{R}^+;H^{m}(\mathbb{R}^n)) $ of the Cauchy problem associated to the Kuznetsov equation such that for all $t\in \mathbb{R}^+$
\begin{equation}\label{controlEninitdata}
\sqrt{E_{\frac{m}{2}}[u](t)}\leq \sqrt{(\frac{3}{2}+c^2)  \left(1+\frac{(2c^2+2)^{m+2}-1}{(2c^2+2)^2-1}\right)}(\Vert \nabla u_0\Vert_{H^m(\mathbb{R}^n)}+\Vert u_1\Vert_{H^m(\mathbb{R}^n)}).
\end{equation}
\end{theorem}
\begin{proof}
We want to show (\ref{controlinitdata}).
 To do it, we  perform the induction on $i\in \lbrace 0;1;...;\frac{m}{2}\rbrace$ proving that   the time derivatives of the solution of the Cauchy problem~(\ref{kuz})--(\ref{ci}) $u$ at $t=0$ satisfy for all $i\in \lbrace 0;1;...;\frac{m}{2}\rbrace$ and for $k\in \mathbb{N}$, $0\leq k\leq i$ the following estimate
\begin{equation}\label{controlnorminitk}
\Vert \partial_t^k u_t(0)\Vert_{H^{m-2k}(\mathbb{R}^n)}\leq a_k(\Vert \nabla u_0\Vert_{H^m(\mathbb{R}^n)}+\Vert u_1\Vert_{H^m(\mathbb{R}^n)}),
\end{equation}
with $a_0=1$, $a_1=2c^2+2$ and
$$a_{k+1}=a_k+2c^2 a_{k-1}+2\sum_{i=0}^k a_i+1 \quad \hbox{for } 1\leq k\leq \frac{m}{2}-1.$$
For $i=0$ the proof is direct.
For $i=1$  from the Kuznetsov equation we have
$$u_{tt}(0)=\frac{1}{1-\alpha \varepsilon u_1}(c^2 \Delta u_0+\nu \varepsilon \Delta u_1+ \beta \varepsilon \nabla u_0\nabla u_1).$$
As for a small enough $\eps$ it holds $\Vert \frac{1}{1-\alpha \varepsilon u_1}\Vert_{\infty}\leq 2$, taking the $\Vert .\Vert_{H^{m-2}(\mathbb{R}^n)}-$norm of the last equality we obtain
\begin{align}\label{controlnorminit1}
\Vert u_{tt}(0) \Vert_{H^{m-2}(\mathbb{R}^n)}\nonumber\leq  2(c^2 & \Vert \Delta u_0\Vert_{H^{m-2}(\mathbb{R}^n)}+\nu \varepsilon \Vert \Delta u_1\Vert_{H^{m-2}(\mathbb{R}^n)}\nonumber\\
&+\beta \varepsilon \Vert\nabla u_0\nabla u_1\Vert_{H^{m-2}(\mathbb{R}^n)}).
\end{align}
Thanks to~\cite{ADAMS-1975} we have
for all $l\in \mathbb{N}$ and for all $k\in \mathbb{N}$, $0\leq l\leq m$ and $0\leq k\leq m-l$
the continuous embedding of the product
\begin{equation}\label{injsobprod}
H^{m-l}(\mathbb{R}^n)\times H^{k+l}(\mathbb{R}^n) \hookrightarrow H^k(\mathbb{R}^n).
\end{equation}
Thus we can write for (\ref{controlnorminit1})
\begin{align*}
\Vert u_{tt}(0)  \Vert_{H^{m-2}(\mathbb{R}^n)}\nonumber \leq 2(c^2 & \Vert \nabla u_0\Vert_{H^{m}(\mathbb{R}^n)}+\nu \varepsilon \Vert u_1\Vert_{H^{m}(\mathbb{R}^n)}\\
&+\beta \varepsilon  K \Vert \nabla u_0\Vert_{H^{m-1}(\mathbb{R}^n)}\Vert \nabla u_1\Vert_{H^{m-1}(\mathbb{R}^n)}),
\end{align*}
and by Young's inequality we find
\begin{align}
\Vert u_{tt}(0)\Vert_{H^{m-2}(\mathbb{R}^n)}\nonumber\leq  2\left[c^2\right. &\Vert \nabla u_0\Vert_{H^{m}(\mathbb{R}^n)}+\nu \varepsilon \Vert u_1\Vert_{H^{m}(\mathbb{R}^n)}\nonumber\\
&\left.+\frac{1}{2}\beta  \varepsilon K\left(\Vert \nabla  u_0\Vert_{H^{m}(\mathbb{R}^n)}^2+ \Vert u_1\Vert_{H^{m}(\mathbb{R}^n)}^2\right)\right].\label{controlnorminit2}
\end{align}
Choosing $\eps$ small enough such that

$$\beta \varepsilon K\Vert \nabla u_0\Vert_{H^{m}(\mathbb{R}^n)}\leq 1,\;\;\;\beta \varepsilon K\Vert  u_1\Vert_{H^{m}(\mathbb{R}^n)}\leq 1,\;\;\; \nu \varepsilon\leq\frac{1}{2},$$  from (\ref{controlnorminit2}) it follows
\begin{align*}
\Vert u_{tt}(0)\Vert_{H^{m-2}(\mathbb{R}^n)}\leq& 2 \left[\left( c^2+\frac{1}{2}\right)\Vert \nabla u_0\Vert_{H^{m}(\mathbb{R}^n)}+ \left(\frac{1}{2}+\frac{1}{2}\right) \Vert u_1\Vert_{H^{m}(\mathbb{R}^n)}\right]\\
\leq& (2c^2+2) (\Vert \nabla u_0\Vert_{H^m(\mathbb{R}^n)}+\Vert u_1\Vert_{H^m(\mathbb{R}^n)}).
\end{align*}
Let define now the induction hypothesis: for $i\in \lbrace 0;1;...;\frac{m}{2}-1\rbrace$ for $k\in \mathbb{N}$, $0\leq k\leq i$ it holds estimate (\ref{controlnorminitk}). Now we want to show it for $i+1$, by the induction hypothesis we just need to show
$$\Vert \partial_t^{i+1} u_t(0)\Vert_{H^{m-2(i+1)}(\mathbb{R}^n)}\leq a_{i+1}(\Vert \nabla u_0\Vert_{H^m(\mathbb{R}^n)}+\Vert u_1\Vert_{H^m(\mathbb{R}^n)}).$$
Deriving $i$-times on time the Kuznetsov equation,  for $i\geq 1$ we obtain
\begin{align*}
\partial^i_t u_{tt}(0)=& \frac{1}{1-\alpha u_1}\Big(c^2 \Delta \partial^i_t u(0) +\nu \varepsilon \Delta \partial^i_t u_t(0)+ \alpha \varepsilon \sum_{k=0}^{i-1} C^k_i \partial^{i-k}_t u_t(0) \partial^k_t u_{tt}(0)\\
& \;\;\;\;\;\;\;\;\; +\beta \varepsilon \sum_{k=0}^{i} C^k_i \nabla \partial^{i-k}_t u(0) \nabla \partial^k_t u_t(0)\Big).
\end{align*}
We take the $\Vert .\Vert_{H^{m-2(i+1)}}-$norm of this equation and in the same way as for $i=1$ we show that
\begin{align*}
 \Vert \partial^{i+1}_t u_{t}(0)&\Vert_{H^{m-2(i+1)}(\mathbb{R}^n)}\\
 \leq & \left(2c^2 a_{i-1}+a_i+2\sum_{k=0}^i a_k+1\right)(\Vert \nabla u_0\Vert_{H^m(\mathbb{R}^n)}+\Vert u_1\Vert_{H^m(\mathbb{R}^n)})\\
 \leq & a_{i+1} (\Vert \nabla u_0\Vert_{H^m(\mathbb{R}^n)}+\Vert u_1\Vert_{H^m(\mathbb{R}^n)}).
\end{align*}
This concludes the induction.

With the induction result we have  for $k\in \mathbb{N}$, $0\leq k\leq \frac{m}{2}$
\begin{align*}
\Vert \partial_t^k u_t(0)\Vert_{H^{m-2k}(\mathbb{R}^n)}\leq a_k(\Vert \nabla u_0\Vert_{H^m(\mathbb{R}^n)}+\Vert u_1\Vert_{H^m(\mathbb{R}^n)}),
\end{align*}
where
$$a_k\leq(2c^2+2)^k.$$
Therefore we can write
\begin{align*}
E_{\frac{m}{2}}[u](0)\leq & \left(1+\sum_{i=0}^{\frac{m}{2}} a_i^2\right)(\Vert \nabla u_0\Vert_{H^m(\mathbb{R}^n)}+\Vert u_1\Vert_{H^m(\mathbb{R}^n)})^2\\
\leq & \left(1+\frac{(2c^2+2)^{m+2}-1}{(2c^2+2)^2-1}\right)(\Vert \nabla u_0\Vert_{H^m(\mathbb{R}^n)}+\Vert u_1\Vert_{H^m(\mathbb{R}^n)})^2.
\end{align*}
Hence, taking the initial data satisfying estimate (\ref{controlinitdata}) we have the following estimate for the initial energy
$$E_{\frac{m}{2}}[u](0)\leq \frac{2\nu^2}{(\frac{3}{2}+c^2)C_m^2 \max(\alpha^2,\beta^2)}.$$
Consequently, by Theorem~\ref{ThDecrEnnu} Point 2 for all $t\in \mathbb{R}^+$ we obtain estimate (\ref{controlEninitdata}).
\end{proof}


\def\refname{References}
\bibliographystyle{siam}
\bibliography{ref}

\end{document}